\documentclass[11pt]{article}
\usepackage{eurosym}
\usepackage{amsfonts}
\usepackage{amsmath}
\usepackage{mathtools}
\usepackage{geometry}
\usepackage{amssymb}
\usepackage{graphicx}
\usepackage{color}
\usepackage[author={Oana}]{pdfcomment}
\usepackage[colorinlistoftodos]{todonotes}
\usepackage{titling}
\usepackage{lipsum}
\usepackage{mathrsfs}
\usepackage{bbm}
\usepackage{dsfont}

\usepackage{geometry}

\usepackage{CJKutf8} %for japanese

\usepackage{graphicx}
\graphicspath{ {./plots/} }

\usepackage{ntheorem}

\newtheorem{theorem}{Theorem}

\newtheorem{definition}[theorem]{Definition}

\newtheorem{lemma}[theorem]{Lemma}

\newtheorem{proposition}[theorem]{Proposition}
\newtheorem{remark}[theorem]{Remark}

\theoremstyle{nonumberplain}
\newtheorem{proof}[theorem]{Proof}

\numberwithin{equation}{section}

\begin{document}

\date{}
\title{\textbf{Global solutions for stochastically controlled fluid dynamics models }}
\author{Dan Crisan \qquad Oana Lang \\
\\ {\small Department of Mathematics, Imperial College London, SW7 2AZ, UK}}
\maketitle
\begin{abstract}
    For a class of evolution equations that possibly have
    only local solutions, we introduce a stochastic component that ensures that the solutions of the corresponding stochastically perturbed equations are global. The class of partial differential equations amenable for this type of treatment includes the 3D Navier-Stokes equation, the rotating shallow water equation (viscous and inviscid), 3D Euler equation (in vorticity form), 2D Burgers' equation and many other fluid dynamics models.         
\end{abstract}
\section{Introduction} 

A large class of partial differential equations (PDEs) exists in the literature for which global existence of solutions does not necessarily hold, with several such examples coming from geophysical fluid dynamics, see e.g. \cite{Bresch1}, \cite{LCRo}, \cite{LinkNguyen}, \cite{LuoZhang}. In this paper we show that the addition of a suitable stochastic term to this type of PDEs ensures existence of a global solution. The stochastic term can be chosen to act on the solution selectively, for example, it can act only when the solution of the PDE becomes too large in a suitably chosen norm. 
The class of PDEs that can be considered in this sense has the form
\begin{equation}\label{PDE}
    dX_t = A_t(X_t)dt
\end{equation}
where $A$ is a (possibly) time dependent nonlinear operator. A representative example of such a PDE is the 3D Navier Stokes equation, either written for the fluid velocity $u$ 
\begin{equation}\label{nseqn}
du+ u\cdot \nabla udt =\Delta udt
\end{equation}
or for the fluid vorticity $\omega$ 
\begin{equation}
d\omega+ u\cdot \nabla \omega dt - \omega\cdot \nabla u dt =\Delta \omega dt. 
\end{equation}
The associated SPDE will have the form 
\begin{equation}\label{SPDE}
    dX_t = A_t(X_t)dt+B_t(X_t)dW_t
\end{equation}
where the stochastic term will be constructed such that it reduces the norm of the system below a suitably chosen value. For the 3D Navier Stokes equation \eqref{nseqn} for instance, a possible choice for the diffusion term is $B(u)=\theta\|u\|^\alpha u$ where $\|u\|$ is a certain Sobolev norm of $u$ and $\alpha,\theta$ are constants that match the nonlinearity of $A$.

The stochastic term stops the blow-up through a "second order" effect. To be more precise the increase of the norm of the solution of the partial differential equation is counteracted by a term 
driven by the quadratic variation of the stochastic term. As it is well-known, the stochastic term does not have finite variation paths. When the chain rule is applied for a given function $\varphi$ 
of the norm of the solution, an additional term appears as a result of the infinite variation of the stochastic term. This term involves the second derivative of $\varphi$. So whilst the nonlinear term appearing in the PDE is multiplied by the first derivative, the quadratic variation term is multiplied by the second derivative. To cancel the two, one needs a function that changes signs when differentiated.
In the literature on blow-up of (one-dimensional) SDEs the function that has this property is called the \textit{scale function}: more precisely, for a one-dimensional SDEs the scale function $\varphi$ is chosen so that the drift term of the process given by the function of the solution, vanishes (see e.g. \cite{KaratzasShreve}). That means that $\varphi$ will satisfy a second order partial differential equation. Once the drift is eliminated, the remaining stochastic term behaves like a Brownian motion with a time clock depending on its quadratic variation. Explicit conditions can then be given to show when the blow-up will occur.

For an infinite-dimensional SDE (in other words, an SPDE), we cannot apply directly this procedure. The reason for this is that, generally the drift term depends not only on the norm that we are seeking to prevent the blow-up, but also on an additional norm of the solution. Depending on the case that we are in, this can be stronger or weaker than the existence norm. We seek a scale function for the norm of the solution of the SPDE. This scale function, denoted by $\varphi$, can no longer be obtained as a solution of a PDE. Instead, we look for a scale function that makes the drift uniformly bounded from above. In the classical case this will mean that the solution of the one dimensional equation might still blow-up to $-\infty$. Here the quantity that we want to control is a norm which, by definition remains nonnegative, hence there is no risk of negative blow-up. 

Once we remove the constraint of having a null drift, we end up with a large class of stochastic noises that can be used. This is still not enough to stop the blow-up of the norm. Whilst the drift term will have an upper bound of the form $mt$ (where $m$ is the norm itself and $t$ is time), the stochastic term (denoted by) $\tilde M$ is a local martingale that might still blow up in finite time. As a result we seek a scale function that will not only bound the drift by $mt$, but it will make it less than $mt-\epsilon \langle\tilde M\rangle_t$, where $\langle\tilde M\rangle$ is the quadratic variation of the local martingale $\tilde M$ and $\epsilon$ is a positive constant. As a result, the scale function of the norm will now be bounded \emph{pathwise} by $mt + \tilde{M}_t-\epsilon \langle\tilde M\rangle_t$, as the term $t\rightarrow \tilde{M}_t-\epsilon \langle\tilde M\rangle_t$ is bounded on the whole interval $[0,\infty )$ by the supremum of a Brownian motion plus a negative constant term, which can be controlled $\mathbb{P}$ - almost surely, see Proposition \ref{revuzyor}.

Whilst not used directly in our analysis, we should also mention that the stochastic term in equation (\ref{SPDE}) is "infinitesimally unbiased" in the sense that it is a local martingale. It is not a proper martingale as we do not know if it is integrable. This is a major hurdle in the analysis of the solution of the SPDE: all terms can only be controlled in probability \emph{and} not in expectation. We cannot guarantee that the norm process has finite second moment or even that it has finite mean. This applies to all steps in the construction of the global solution of the equation (\ref{SPDE}): the a priori bounds, the tightness results and the convergence of the solutions of the approximating equation.                

The results presented below can be used to introduce a (stochastic) control for a dynamical system modelled by the solution of the PDE. We describe just the intuitive idea, and refer to e.g. \cite{Bensoussan} for more details. The application of the control is done only when the norm of the solution of the PDE exceeds a certain level. This is possible as the scale function is chosen not only to make the drift term less than $mt-\epsilon \langle \tilde M \rangle_t$, but in fact less than $-\epsilon \langle \tilde M \rangle_t$ for values of the norm larger that a suitable chosen constant $K$. This permits the following control strategy: 
We denote by $\tau_0$ the first (deterministic) time when the norm of the solution of the PDE reaches level  $\varphi^{-1}(2K)$, where $K$ is the constant defined above.  If $\tau_0=\infty$, then no control is needed. However if $\tau_0<\infty$, we apply the stochastic control to the PDE. We keep the control until the norm of the solution, goes down to level $\varphi^{-1}(K)$. We denote this (stopping) time by $\rho_0$. Once the norm reaches level $K$ we remove the stochastic control and let the equation evolve deterministically until again it reaches level $\varphi^{-1}(2K)$ (denote the new stopping time by $\tau_1$) at which point we apply the stochastic control again until the solution reaches level $\varphi^{-1}(K)$ (denote the new stopping time by $\rho_1$) and repeat the procedure. The resulting process is a (strong) Markov process built by piecing together deterministic trajectories and stochastically controlled ones. The general definition of the (stopping) times is as follows  
\begin{equation}
\tau_i=\inf \{t\ge \rho_{i-1} | \ \ ||X_t||\ge \varphi^{-1}(2K) \}, \ \ \ 
\rho_i=\inf \{t\ge \tau_{i} | \ \ ||X_t||\le \varphi^{-1}(K) \}
\end{equation}                 
By convention, we choose $\rho_{-1}:=0$. As explained above, the system runs according to the original PDE (\ref{PDE}) in the intervals $[\rho_{i-1}, \tau_i]$ and according to the SPDE (\ref{SPDE}) on the 
intervals $[\tau_i, \rho_{i}]$. Note that in all case below, it takes time for the norm of the solution of the PDE to reach level $\varphi^{-1}(2K)$ starting from level $\varphi^{-1}(K)$. More precisely, one can show that there exists a positive constant $\alpha$ which depends only on $K$ such that $\tau_i - \rho_{i-1} >\alpha$. This ensures that $\lim_{i\rightarrow \infty}\tau_i=\lim_{i\rightarrow \infty}\rho_i=\infty$. This ensures that the process $X$ is defined on the whole time line $[0,\infty]$.              

Also note that, in the first two cases, once the stochastic control is introduced at time $\tau_i$, that norm of the process will reach level $\varphi^{-1}(K)$ in finite time. In other words, $\rho_i-\tau_i<\infty$, $P-a.s.$. As stated above, observe that on the interval $[\tau_i,\rho_i]$ the upper bound for the process $t\rightarrow \varphi (\|X_t\|)$ is the process $t\rightarrow \varphi(2K)+\tilde M_t-\epsilon [\tilde M]_t $. This process is, in turn, a time-changed Brownian motion with negative drift which will eventually hit level $\varphi(K)$ as the time change tends to infinity    

In the third case, it may be possible for $\rho^i$ to be $\infty$. The fix in this case is to choose a higher level $K$ which will ensure that the drift term of the scale function applied to the norm $t\rightarrow \varphi (\|X_t\|)$ has an upper bound  of the form $t\rightarrow \varphi(2K)-\delta(t-\tau_i) + \tilde M_t-\epsilon [\tilde M]_t$ which will in turn ensures $t\rightarrow \varphi (\|X_t\|)$ eventually reaches level $\varphi(K)$.
      
We remark that the blow-up control obtained through the stochastic term does not act in a way similar to a forcing term. More precisely, it is the quadratic variation of the stochastic term that controls the 
blow-up. In some sense this can be viewed as a second order effect as the stochastic term does not directly influence the blow-up. The system becomes infinitesimally unbiased due to the quadratic variation.
To understand this, let us look at the following linear one-dimensional ode (we write it in differential term as we will add a stochastic term to it:
\[df_t=af_t dt, \ \ \ \ f_0,a>0. \]
The solution of this ODE does not blow-up, but it does increase exponentially fast in time: $f_t=f_0\exp(at)$. Let us contrast this with the solution of the same equation  but now stochastically perturbed: 
\[df_t=af_t dt+bf_t dW_t, \ \ \ \ f_0,a>0, \]
where we get a one dimensional geometric Brownian motion, with the explicit solution 
$$f_t=f_0\exp\left(at-{b^2\over 2}[W]_t+bW_t\right)=\exp\left((a-{b^2\over 2})t+bW_t\right),$$  
where $\langle W \rangle=t$ is the quadratic variation of the Brownian motion $W$. We deduce immediately that, as soon as the diffusion coefficient is sufficiently high (in other words as soon as we inject sufficient noise into the system), more precisely, if we choose $b^2>2a$ then the solution converges to $0$ asymptotically. The effect of the stochasticity is severe: from exponential increase we end up with exponential decay to $0$. However we should note that the expected value of the geometric Brownian motion still increases exponentially (in fact it  satisfies the deterministic PDE - the perturbation is unbiased) even though, with probability 1,  every path decays to 0.    

We should further remark that the control we need is more severe than the exponential increase described above. More precisely we need to control a (possible) blow-up in finite time. As a result we cannot expect to be able to use a linear/multiplicative noise, we will need instead a super-linear noise. This will have several secondary effects:

\begin{itemize}
\item The resulting stochastic term will \textbf{not} be a martingale but just a local martingale. We cannot even prove that it has finite expectation. As a result, all uniform controls, tightness results and convergence results will have to be adapted to cope with this difficulty. The standard tightness criteria need to be modified to account for these difficulties. 

\item The stochastic term will be nonlinear and so we will not be able to say that the perturbation is unbiased. 

\item The control is not unique: as long as the superlinearity is sufficiently high, we obtained the desired result.

\end{itemize}

Several attempts have been made to show improved properties of SPDEs driven by transport noise, compared to their deterministic counterparts (see e.g. \cite{Gerencser}, \cite{Catellier}, \cite{Flandoli}, \cite{Theresa}). However, due to its intrinsic structure, transport noise seems to leave the turbulent behaviour of the fluid untouched. It becomes therefore interesting to look at different noise structures which could better explain the uncertainty in a turbulent dynamical system. We make a step forward by proposing a noise structure which can at least improve the analytical properties of the system. The physical implications of this type of noise are perhaps less clear. 
Our analysis requires a fine balance between the conditions imposed on the nonlinear operator, the initial conditions and the choice of the stochastic term. Only this way we succeed to show a priori bounds on the solution of the SPDE, the tightness of the approximating sequence and its convergence. The results we obtain can cover a wide class of models from stochastic fluid dynamics: compressible \textbf{and} incompressible, viscous \textbf{and} inviscid.

Formally, we work with four separable Hilbert spaces $\mathscr{D}, \mathscr{F}_1, \mathscr{F}_0, \mathscr{G}$ which are compactly embbeded in each other 
\begin{equation}%\label{embeddings}
    \mathscr{D} \hookrightarrow \mathscr{F}_1 \hookrightarrow \mathscr{F}_0 \hookrightarrow  \mathscr{G}.  
\end{equation}
Typically, the main space of existence of the corresponding solution is $\mathscr{F}_0$, endowed with the uniform norm, when a higher order integral of the $\mathscr{F}_1$ norm also exists (this setting is characteristic to models which contain a viscous term). The space $\mathscr{D}$ is a higher order space in which we can control the uniform norm of an approximating sequence of solutions, while the space $\mathscr{G}$ is the space in which the approximating sequence of solutions converges to the solution of the original equation. 
We identify three distinct cases, depending formally on the space where the initial condition lives and the form of the noise term, see \eqref{caseI}-\eqref{caseIII} in Section \ref{sect:setting}. However, this structure is actually dictated by the form of the nonlinear drift and the required regularity of the solution. The second case (see \eqref{caseII}) corresponds to a stochastic closure of a compressible inviscid model and it is, from this perspective, the most difficult. This is natural, as this is the case where the noise truly improves the properties of the solution - the associated deterministic system cannot be closed without adding artificial viscosity (or a different type of transformation) and to the best of our knowledge, no global existence result exists for such a system. So by adding the right type of stochasticity, we overcome this long-standing issue in the theory of compressible inviscid fluid dynamics models.

It is, a priori, not clear how "physical" is the noise we propose. We know that there is usually a delicate balance between the energy which is created within a multi-scale nonlinear system and the energy which is dissipated through the diffusion term. The blow-up occurs when this balance is broken, and our noise term is designated in a special way to provide an exact balance between absorbed and dissipated energy. That is, the stochastic term provides "the right amount of energy" for the solution to be stabilised. 
Note that the model we present is nonlocal, the noise depending on the whole domain. 

The interest in proving existence of global strong solutions has been continuously growing over the last years. In some sense, our result is closest in spirit to \cite{WuHu} where the authors prove existence of global solutions in a finite-dimensional setting: we extend this result to the infinite-dimensional case, and adapt it to a much more abstract functional setting. In \cite{Maurelli} it has been proven that an explicit Stratonovich-type noise can prevent blow-up: this holds on $\mathbb{R}^d$ with $d \geq 2$, and under specific super-linear growth assumptions on the drift. In \cite{MaurelliNEW} the theory from \cite{Maurelli} is extended such that the existence of a global-in-time strong solution for stochastic Euler-type equations in 2D and 3D can be shown. In \cite{Hao1}
the authors prove that for Camassa-Holm-type equations, a strong enough noise can prevent the blow-up with probability 1, with a specific dependence on the initial conditions. In \cite{Hao2} it has been shown that uniqueness and almost sure existence of solutions can be proven for a one dimensional transport equation under random perturbations on the real line.

Alternative approaches for showing existence of a global solution involve so-called regularisation/$\alpha$-regularisation methods: see e.g. \cite{OlsonTiti1}, \cite{BarbatoMorandinRomito}, \cite{Galeati}. However, we are not going in this direction, as our aim is to find a noise term which makes the solution global, not to prove scaling limits or to prove existence of a global solution for a deterministic equation. 

Finally let us explain the connection with the authors' previous work:  In \cite{LangCrisanMemin} we have shown existence of a global \emph{weak} (in PDE sense) solution, while the existence of a \emph{strong} solution could only be proven locally. In \cite{CrisanLangSALTSRSW}, on the other hand, we have proven that a global strong solution can exist with positive probability. That is, if $\mathcal{T}$ is the corresponding maximal time of existence, we have $\mathbb{P}(\mathcal{T} = \infty) > 0$. However, in this paper we show that in such a case we have $\mathbb{P}(\mathcal{T} = \infty) = 1$.

\section{Setting and assumptions}\label{sect:setting}
Let $(\Omega, \mathcal{F}, (\mathcal{F}_t)_t, \mathbb{P})$ be a fixed stochastic basis. Let $\mathscr{D}, \mathscr{F}_1, \mathscr{F}_0, \mathscr{G}, \mathscr{U}$ be  separable Hilbert spaces such that 
\begin{equation}\label{embeddings}
    \mathscr{D} \hookrightarrow \mathscr{F}_1 \hookrightarrow \mathscr{F}_0 \hookrightarrow  \mathscr{G}  
\end{equation}
are compact embeddings with
\begin{equation}\label{condscproduct}
{}_{\mathscr{D}}\langle a, b\rangle_{\mathscr{G}} = \langle a, b \rangle_{\mathscr{F}_i} \quad \hbox{for all} \quad a\in\mathscr{D}, b\in \mathscr{F}_i.
\end{equation}
and there is $C>0$ and $m\in(0,1)$ such that
\begin{equation*}
\|f\|_{\mathscr{F}_0} \leq C\|f\|_{\mathscr{F}_1}^{m}\|f\|_{\mathscr{G}}^{1-m} \quad \quad \hbox {for all} \quad f\in\mathscr{F}_1
\end{equation*}
We introduce the following operators
\begin{equation}
    A:[0,T] \times \mathscr{F}_{i} \times \Omega \rightarrow \mathscr{G}
\end{equation}
\begin{equation}
    B: [0,T] \times \mathscr{F}_{i} \times \Omega \rightarrow L_2(\mathscr{U},\mathscr{F}_i)
\end{equation}
which are progressively measurable with respect to $\mathcal{B}([0,t])\times\mathcal{B}(V)\times\mathcal{F}_t$ for any $t\in[0,T]$ and $T\in[0,\infty)$ fixed. Above we have either $i=0$ or $i=1$, depending on the structure of the drift, see below for more details.

Consider the following SPDE on $\mathscr{G}$\footnote{By $A_t(X_t)$ and $B_t(X_t)$ we mean the map $\omega \rightarrow A(t, X_t(x,\omega))$ and $\omega \rightarrow B(t, X_t(x,\omega))$}
\begin{equation}\label{eq:maingeneral}
    dX_t = A_t(X_t)dt + B_t(X_t)dW_t 
\end{equation}
where $W$ is a $\mathscr{U}$-valued Wiener process and $X_0$ is an  $\mathscr{F}_i$-valued and $\mathcal{F}_0$-measurable random variable. We assume that this equation admits a unique maximal strong solution $(X, \mathcal{T})$ in $\mathscr{F}_i$ in the sense described in Section \ref{sect:mainresults}. 
We define
\begin{equation}\label{noiseterm}
B_t(X_t):= \theta \displaystyle\|X_t\|_{\mathscr{F}_{i}}^{\alpha_{i}}X_t, 
\end{equation}
where $\alpha_{i}, \theta \in \mathbb{R}_{+}$, and $i$ will be either 0 or 1 depending the conditions imposed on the operator $A$.
Then equation \eqref{eq:maingeneral} 
takes the following form:

\begin{equation}\label{eqnew:maingeneralnew}
    dX_t = A_t(X_t)dt + \theta \displaystyle\|X_t\|_{\mathscr{F}_{i}}^{\alpha_i}X_tdW_t. 
\end{equation}
Let $(\rho_i)_i$ be an orthogonal basis in $\mathscr{D}$ and denote by $\mathscr{D}^d$ the span generated by $\rho_1, \rho_2, \ldots, \rho_d$. So
\begin{equation}
\mathscr{D}^d = span\{\rho_1, \rho_2, \ldots, \rho_d\}.
\end{equation}
We denote by $T_d$ the projection operator $T_d : \mathscr{G} \rightarrow \mathscr{D}^d$.\footnote{So $T_d\mathscr{G} = \mathscr{D}^d$.} Then denote by $A^d$ and $B^d$ respectively the projection of the operators $A$ and $B$. That is
\begin{equation}
\begin{aligned}
& A^d(X^d) := T_d(A(X^d)) \\
& B^d(X^d) := T_d(B(X^d))
\end{aligned}
\end{equation}
with $A^d, B^d:\mathscr{D}^d \rightarrow \mathscr{D}^d$.
We assume that the embedding $\mathscr{D} \hookrightarrow \mathscr{F}$ is compact, continuous, and dense, the embedding $\mathscr{F} \hookrightarrow \mathscr{G}$ is continuous and dense, and $span\{(\rho_i)_i\}$ is dense in $\mathscr{G}$.
So $\mathscr{D}$ is a dense subspace of $\mathscr{G}$. A global solution for equation \eqref{eq:maingeneral} will be constructed using a $d$-dimensional projected equation:

\begin{equation}\label{eqnew:approxgeneralnew}
    dX_t^d = A_t^d(X_t^d)dt + \theta \displaystyle\|X_t^d\|_{\mathscr{F}_{i}}^{\alpha_{i}}X_t^d dW_t^d \quad \hbox{on} \quad [0,T]
\end{equation}
where $X_t^d$ is the classical $d$-dimensional (Galerkin) projection of $X_t$ \footnote{For further standard details related to such a projection see e.g. [Robinson] Section 4.1}. Equation \eqref{eqnew:approxgeneralnew} admits a global solution: this is a direct consequence of Theorem 2.4. and Theorem 2.5. in \cite{WuHu}. An alternative argument is based on the fact since equation  (\ref{eqnew:approxgeneralnew}) has locally Lipschitz coefficients, it has a unique 
maximal solution. In other words, (\ref{eqnew:approxgeneralnew}) has a unique solution $X^d$ defined up to a (maximal) stopping time $\tau_{max}$. In this case $\tau_{max}=\infty$. Should $\tau_{max}<\infty$ it would follow that the solution blows-up, but this is not possible,
as under the conditions imposed below, the corresponding a priori estimates are $d$-independent (see e.g. Section \ref{sect:unifcontrol}), and then
\begin{equation*}
\displaystyle\lim_{\ell\rightarrow\infty} \mathbb{P}\left( \tau_{\ell} < t\right) = 0
\end{equation*}
for any local solution which exists up to the corresponding stopping time $\tau_{\ell}$ and such that $\displaystyle\lim_{\ell \rightarrow \infty}\tau_{\ell} = \tau_{max}$.
This is also the basis of the arguments presented in \cite{WuHu}.
\vspace{4mm}
\begin{definition}{Definition of solutions}
Let $\mathcal{S} = \left(\Omega, \mathcal{F}, (\mathcal{F}_t)_t, \mathbb{P}, (W_t)_t\right)$ be a fixed stochastic basis. 
\begin{itemize}
\item[a.] A \underline{pathwise local analytically strong solution} to \eqref{eq:maingeneral} is given
by a pair $(X,\tau )$ where $\tau :\Omega \rightarrow \lbrack 0,\infty ]$ is
a strictly positive bounded stopping time, $X=(X_t)_t$ is an $\mathscr{F}_i$-valued process 
such that 
$X_{\cdot\wedge \tau }$ is $\mathcal{F}%
_{t}$-adapted for any $t\geq 0$
and \eqref{eq:maingeneral} is satisfied locally i.e.%
\begin{equation*}
X_{t\wedge \tau }=X_{0}+\int_{0}^{_{t\wedge \tau }}A_s(X_{s})
ds+\int_{0}^{_{t\wedge \tau }}B_s(X_{s}) dW_{s}
\end{equation*}%
holds $\mathbb{P}$-almost surely as an identity in $\mathcal{G}$.

\item[b.] If $\tau =\infty $ then the solution is called 
\underline{global}.

\item[c.] A pathwise \underline{maximal solution} to \eqref{eq:maingeneral} is given by a pair $(X, \mathcal{T})$ where $\mathcal{T}: \Omega \rightarrow [0, \infty]$ is a non-negative stopping time and $X=(X_t)_t$ is an $\mathscr{F}_i$-valued process 
for which there exists an increasing sequence of stopping times $(\tau^n)_n$ with the following properties: 

\begin{itemize}
\item[i.] $\mathcal{T}=\lim_{n\rightarrow \infty }\tau ^{n}$ and $\mathbb{P}(%
\mathcal{T}>0)=1$
\item[ii.] $(X,\tau ^{n})$ is a pathwise local solution to \eqref{eq:maingeneral}
for every $n\in \mathbb{N}$
\item[iii.] if $\mathcal{T}<\infty $ then 
\begin{equation*}
\displaystyle\limsup_{t \rightarrow \mathcal{T}} \| X_{t}\|_{\mathscr{F}_0}=\infty .
\end{equation*}
\end{itemize}
%\OL{Maybe explain below why taking the $\mathscr{F}_0$ norm is enough.}
\item[d.] A \underline{global analytically weak solution} to \eqref{eq:maingeneral} is given by an $(\mathcal{F}_t)_t$-adapted process $X=(X_t)_t$
%$X:\Omega \times [0,\infty) \times \mathbb{T}^2 \rightarrow \mathbb{R}^3$ 
such that the following identity holds for any test function
$\varphi \in \mathscr{D}$: 
%$\varphi \in \mathscr{D}=C^{\infty}(\mathbb{T}^2)$: 
\begin{equation*}
  X_t(\varphi)=  \langle X_t, \varphi\rangle = \langle X_0,\varphi\rangle + \displaystyle\int_0^t {}_{\mathscr{G}}\langle A_s(X_s),\varphi\rangle_{\mathscr{D}} ds + \displaystyle\int_0^t {}_{L_2(\mathscr{U},\mathscr{F}_i)}\langle B_s(X_s),\varphi \rangle_{\mathscr{D}} dW_s.
\end{equation*}
\item[e.] The process $X=(X_t)_t$ is a \underline{martingale solution} to \eqref{eq:maingeneral} if there exists a filtered probability space $\left(\tilde{\Omega}, \tilde{\mathcal{F}}, (\tilde{\mathcal{F}}_t)_t, \tilde{\mathbb{P}}\right)$ with a standard Wiener process $\tilde{W}=(\tilde{W}_t)_t$ and an initial condition $\tilde{X}_0$ defined on it, such that $\tilde{X}_0$ and $X_0$ have the same distribution and $X_t$ is an $(\tilde{F}_t)_t$-adapted solution on $\tilde{\Omega}$, to the following equation
\begin{equation}
 dX_t = A_t(X_t)dt + B_t(X_t)d\tilde{W}_t
\end{equation}
with initial condition $\tilde{X}_0$. 
\end{itemize}
\end{definition}

\begin{remark}
Note that an analytically weak solution $X$ becomes analytically strong provided
\begin{equation}
\begin{aligned}
& \mathbb{P}\left( \displaystyle\int_0^T \|A_s(X_s)\|_{\mathscr{F}_i}ds<\infty \right) = 1 \\
& \mathbb{P}\left( \displaystyle\int_0^T \|B_s(X_s)\|_{L_2(\mathscr{U},{\mathscr{F}_i)}}^2ds<\infty \right) = 1
\end{aligned}
\end{equation}
\end{remark}

Our main result says, informally, that\\

\textit{Under the above assumptions and given the embeddings \eqref{embeddings}, equation \eqref{eqnew:maingeneralnew} admits a global strong solution in $\mathscr{F}_0$, $\mathscr{F}_1$ and 
$\mathscr{D}$, respectively. } 

\vspace{3mm}
The precise statement of this results is given in Theorem \ref{thmnew:maintheorem} below. We distinguish three main cases:

\noindent\textbf{Case I:} The initial condition $X_0\in \mathscr{F}_0$, $\theta >0$, and 
\begin{equation}\label{caseI}
B_t(X_t) = \theta \|X_t\|_{\mathscr{F}_0}^{\alpha_0}X_t.
\end{equation}
\noindent\textbf{Case II:} The initial condition 
 $X_0 \in \mathscr{D}$, $\theta >0$, and 
\begin{equation}\label{caseII}
B_t(X_t) = \theta\|X_t\|_{\mathscr{F}_1}^{\alpha_1}X_t. 
\end{equation}
\noindent\textbf{Case III:} The initial condition $X_0 \in \mathscr{F}_1$, $\theta >0$, and 
\begin{equation}\label{caseIII}
B_t(X_t) = \theta \|X_t\|_{\mathscr{F}_0}^{\alpha_0}X_t.
\end{equation}

\begin{remark}
Typical models which can be covered by Case I are compressible and incompressible viscous models (e.g. viscous rotating shallow water models), while typical models covered by Case II and Case III are compressible inviscid (e.g. inviscid rotating shallow water model) and, respectively, incompressible models (e.g. Euler equation). For explicit details regarding these connections and further examples, see Section \ref{sect:applications}. 
\end{remark}
\begin{remark}
We emphasize that the role of the space $\mathscr{F}_1$ in Case III is, in a certain sense, different from the role of the same space in the first two cases and it should be read in the following key: in Case III $\mathscr{F}_1$ \textbf{is} the main space of existence of the solution, and the noise depends on the norm of the solution in a larger space, that is the $\mathscr{F}_0$ norm. We could choose the space of existence to be $\mathscr{F}_0$ as in the previous two cases, but then we would need to introduce a new larger space between $\mathscr{F}_0$ and $\mathscr{G}$, say $\mathscr{F}^{\star}$ which would be used to construct the stochastic term. The analysis is identical in any of these two situations, so we decided not to introduce an extra space, and to proceed with the analysis in $\mathscr{D}\hookrightarrow \mathscr{F}_1 \hookrightarrow \mathscr{F}_0 \hookrightarrow \mathscr{G}$. Note that in order to stop the blow-up in a generic functional space $\mathscr{F}^{\star}$ we need to make sure that the evolution of the corresponding solution can first be controlled in a smaller/higher order space.
\end{remark}
\begin{remark}
It must be emphasized that for the compressible inviscid model we cannot replace the control in $\mathscr{F}_0$ with control in $\mathscr{D}$, as for showing control in $\mathscr{D}$ we need control in $\mathscr{F}_0$. That is, overall we need all four spaces $\mathscr{D}, \mathscr{F}_1, \mathscr{F}_0$ and $\mathscr{G}$, and we use both Proposition \ref{prop:unifcontrol} and Proposition \ref{prop:unifcontrolinD} to ensure the required uniform control.
\end{remark}

In order to obtain global existence, the following assumptions $(A)$ will need to be satisfied by the d-dimensional approximation $(X^d)_d$:\\
\begin{itemize}
\item \textbf{Assumption (A1):} The initial condition is in $\mathscr{F}_0$ and for $a\in\mathcal{D}^d$ it holds that
\begin{equation}\label{cond:compressviscous}
\langle a, A^d(a)\rangle_{\mathscr{F}_0} \leq C_1 \|a\|_{\mathscr{F}_0}^{\gamma_1} -C_2\|a\|_{\mathscr{F}_1}^{2}
\end{equation}
where $C_1, \gamma_1$ are positive constants and $C_2 \in \mathbb{R}$.

For $a\in \mathscr{F}_1$
\begin{equation}
\|a\|_{\mathscr{F}_0} \leq C_0\|a\|_{\mathscr{F}_1}^{\alpha}\|a\|_{\mathscr{G}}^{\beta}
\end{equation}
and
\begin{equation}\label{cond:tight1}
\|A^d(a)\|_{\mathscr{G}} \leq C_1\|a\|_{\mathscr{F}_0}^{\gamma^1}+
C_2\|a\|_{\mathscr{F}_1}^{\gamma^2}+C_3,
\end{equation}
where $\alpha, \beta, \gamma^i, C_i, i\in \{1,2,3\}$ are positive constants. 
Moreover, for any $a,b \in \mathscr{F}_0$
\begin{equation}\label{cond:convA1}
\|A^d(a) - A^d(b)\|_{\mathscr{G}} \leq C(1+\|a\|_{\mathscr{F}_1}+ \|b\|_{\mathscr{F}_1})\|a - b\|_{\mathscr{F}_0}.
\end{equation}
with $C$ a positive constant. 
\item \textbf{Assumption (A2):} The initial condition is in $\mathscr{D}$, \eqref{cond:compressviscous}, \eqref{cond:convA1}, and the following conditions hold:
\begin{equation}\label{cond:compressinviscid}
\langle a, A^d(a)\rangle_{\mathscr{D}} \leq C_1 \|a\|_{\mathscr{F}_1}^{\gamma_1}\|a\|_{\mathscr{D}}^{\gamma_2}
\end{equation}

\begin{equation}
\|a\|_{\mathscr{F}_1} \leq C\|a\|_{\mathscr{D}}\|a\|_{\mathscr{G}}
\end{equation}
where $C_1, C, \gamma_1, \gamma_2$ are positive constants. Condition \eqref{cond:tight1} holds with $C_2=0$.
We also assume that the approximate equation \eqref{eqnew:approxgeneralnew} admits a  solution in $\mathscr{F}_0$ which does not hit zero or, if it hits zero, it remains zero. 
\item \textbf{Assumption (A3):} The initial condition is in $\mathscr{F}_1$ and the following conditions hold:
\begin{equation}\label{cond:incompress}
\langle a, A^d(a)\rangle_{\mathcal{F}_1} \leq C_1 \|a\|_{\mathscr{F}_0}^{\gamma_{13}}\|a\|_{\mathscr{F}_1}^2
\end{equation}
%For $a\in\mathscr{F}_1$
\begin{equation}
\|a\|_{\mathscr{F}_0} \leq C\|a\|_{\mathscr{F}_1}\|a\|_{\mathscr{G}}
\end{equation}
where $C_1, C, \gamma_1, \gamma_2, \gamma$ are positive constants. Condition \eqref{cond:tight1} holds with $C_1=0$.
Moreover, for any $a,b \in \mathscr{F}_0$
\begin{equation}\label{cond:convA3}
\|A^d(a) - A^d(b)\|_{\mathscr{G}} \leq C(1+\|a\|_{\mathscr{F}_1}+ \|b\|_{\mathscr{F}_1})\|a - b\|_{\mathscr{F}_1}.
\end{equation}
with $C$ a positive constant. 
\end{itemize}

\begin{remark}$\left.\right.$\\[-5mm]
\begin{itemize}
\item [a.] In practice, condition \eqref{cond:compressviscous} can be obtained directly from a condition of the form
\begin{equation}\label{Bcompressviscousgenold}
{}_{\mathscr{D}}\langle a, A^d(a)\rangle_{\mathscr{G}} \leq C_1 \|a\|_{\mathscr{F}_0}^{\gamma_1} + C_2\|a\|_{\mathscr{F}_1}^{\gamma_2}
\end{equation}
which is natural for a nonlinear drift which contains a higher order term (and frequently with $\gamma_2=2$). The equivalence of the scalar products $\langle \cdot,\cdot \rangle_{\mathscr{F}_0}$ and ${}_\mathscr{D}\langle \cdot, \cdot \rangle_{\mathscr{G}}$ is due to condition \eqref{condscproduct}.
\item [b.] We highlight that one does not necessarily need an extra space in case II but then we need to go to a smaller space in order to have convergence of the stochastic integral, and in that case the initial condition needs to be in $\mathscr{D}$ which is more restrictive. 
\item [c.] It is  noteworthy that conditions such as \eqref{cond:compressviscous}, \eqref{cond:compressinviscid}, and \eqref{cond:incompress} are natural for geophysical fluid dynamics models. Condition \eqref{cond:incompress}, for instance, can be verified for an incompressible model: the fact that the $\mathscr{F}_1$-norm appears on both sides is possible due to the incompressibility condition. Unlike that, condition \eqref{cond:compressinviscid} can be verified for a compressible inviscid model, but the fact that the $\mathscr{D}$-norm appears on both sides is possible due to the finite-dimensionality of the space $\mathscr{D}^d$ which is dense in $\mathcal{D}$ and which is the space of existence for the finite-dimensional approximating system. Condition $\eqref{cond:compressviscous}$ is, in some sense, typical for compressible viscous models with $\mathcal{F}_0$-valued solutions (see particular cases in e.g. \cite{LangCrisanMemin} and \cite{CrisanLangSALTSRSW}). We call \textit{"space of existence"} the space where the uniform norm is controlled. This is just a shortened terminology we use here, as one can notice that the full functional spaces where the respective solutions live is 
$L^2\left(\Omega; C\left((0,T), \mathscr{I}\right)\right)$
where $\mathscr{I}$ is equal to  $\mathscr{F}_0$, $\mathscr{F}_1$ or
$\mathscr{D}$ endowed, respectively, with the following norms:
\begin{equation*}
\begin{aligned}
&\hbox{\textbf{Case I}} \quad \quad \quad  ||| X |||_{\mathscr{F}_{0},\mathscr{F}_{1},T}^{2}:=\sup_{t\in \left[ 0,T\right]}\|X_{t}\|_{\mathscr{F}_{0}}^2+\int_{0}^{T}\|X_{t}\|_{\mathscr{F}_{1}}^{2}dt\\
&\hbox{\textbf{Case II}} \quad \quad \quad   |||X|||_{\mathscr{D},T}:=\sup_{t\in \left[ 0,T\right] }\|X_{t}\|_{\mathscr{D}}\\
&\hbox{\textbf{Case III}} \quad \quad \quad |||X|||_{\mathscr{F}_{1},T}:=\sup_{t\in \left[ 0,T\right] }\|X_{t}\|_{\mathscr{F}_{1}}.
\end{aligned}
\end{equation*}
\item [d.] In condition \eqref{cond:compressviscous} we have $C_2 > 0$ for compressible viscous models such as viscous rotating shallow water (Case I) and $C_2 < 0$ for compressible inviscid models such as standard (inviscid) rotating shallow water (Case II). This is because in the second case there is no dissipation in the original deterministic drift. Instead, we "create" dissipation from the stochastic term, to counter-balance the effect of the nonlinear term. In order to further ensure global existence, the intrinsic properties of the stochastic term are instrumental (due to e.g. \eqref{gawareckimandrekar} or \eqref{revuzyor}).
\item [e.] Under similar conditions, in a sequel to this work we will show that the (time-homogeneous
version of the) process $X$ which solves the equation has an invariant measure on $\mathscr{G}$. To do this we will show it is a Markov process and
that its corresponding semigroup $\left\{ P_{t},t\geq 0\right\} $ is Feller.
Moreover, for an arbitrary $\mu $ on $\mathscr{G}$ we define the set of
measures 
\begin{equation*}
\mu _{t}(A)=\frac{1}{t}\int_{0}^{t}\int_{\mathscr{G}}P(s,x,A)ds\mu
(dx),~~t\geq 0
\end{equation*}%
for all $A\in \mathscr{B}(\mathscr{G})$. Then we will show that $\left\{ \mu
_{t}\right\} _{t\geq 0}$ is tight and that any weak limit of a subsequence $%
\left\{ \mu _{t_{n}}\right\} _{t_{n}\geq 0}$ such that $\left\{
t_{n}\right\} _{n=1}^{\infty }\subset \mathbb{R}_{+},t_{n}\rightarrow \infty 
$ is an invariant measure for $X$. The key result used to prove tightness is
the control given in Proposition \ref{prop:unifcontrol}.
\item [f.] %{\color{blue} 
Note that for showing uniqueness we would need a locally Lipschitz condition such as
\begin{equation}\label{cond:convgen}
\|A^d(a) - A^d(b)\|_{\mathscr{G}} \leq C(\|a\|_{\mathscr{F}_1}, \|b\|_{\mathscr{F}_1})\|a - b\|_{\mathscr{G}}
\end{equation}
where $C$ is continuous as a function of $a$ and $b$ and it is independent from the dimension $d$.
This implies uniqueness of the maximal solution (which we have already assumed) and then we can deduce uniqueness of the newly-constructed global solution: we know that any maximal solution $\hat{X}$ is unique up to a maximal stopping time $\mathcal{T}$. That is $\hat{X} = X$ on $[0,\mathcal{T})$. However, we show that $X$ itself is a \textbf{global} solution. As a consequence, we need to have $\mathbb{P}\left( \mathcal{T} = \infty \right)=1$.
Note also that condition \eqref{cond:convgen} is a generalised version of \eqref{cond:convA1} and \eqref{cond:convA3}, so assumptions $(A1)-(A3)$ imply directly uniqueness for the maximal solution.
\end{itemize}
\end{remark}

\section{Main results}\label{sect:mainresults}
Our main theorem is given below: 
\begin{theorem}\label{thmnew:maintheorem}
\label{thmnew:maintheorem-v2} Under assumptions $(A1)-(A3)$, the
stochastic partial differential equation \eqref{eqnew:maingeneralnew} admits
a global strong solution $X$ with paths that belong to the space $C\left(
[0,\infty ),\mathscr{G}\right) ,$ $\mathbb{P}$-almost surely, such that for any $T>0$%
, 
\begin{equation}
\mathbb{P}\left( \sup_{t\in \left[ 0,T\right] }\left\vert \left\vert X_{s}\right\vert
\right\vert _{\mathscr{G}}<\infty \right) =1
\end{equation}
\end{theorem}

\begin{remark}
As we shall see, in each of the three cases, one can prove stronger bounds
for the process $X$ that hold for any $T>0$. They are as follows: 
\begin{eqnarray}
&&\mathbf{Case \ I: \quad}\mathbb{P}\left( \sup_{t\in \left[ 0,T\right] }\left\vert
\left\vert X_{s}\right\vert \right\vert _{\mathscr{F}_{0}}+\int_{0}^{T}\left%
\vert \left\vert X_{s}\right\vert \right\vert _{\mathscr{F}%
_{1}}^{2}ds<\infty \right)  =1.  \label{c1bd} \\
&&\mathbf{Case \ II: \quad}\mathbb{P}\left( \sup_{t\in \left[ 0,T\right] }\left\vert
\left\vert X_{s}\right\vert \right\vert _{\mathscr{D}}<\infty \right)  =1.
\label{c2bd} \\
&&\mathbf{Case \ III: \quad}\mathbb{P}\left( \sup_{t\in \left[ 0,T\right] }\left\vert
\left\vert X_{s}\right\vert \right\vert _{\mathscr{F}_{1}}<\infty \right) 
=1.  \label{c3bd}
\end{eqnarray}
Following from the above controls one can deduce, in each of the three
cases, that 
\begin{equation}
\mathbb{P}\left( \int_{0}^{T}\left\vert \left\vert A_{s}(X_{s})\right\vert \right\vert
_{\mathscr{G}}ds+\int_{0}^{T}\left\vert \left\vert B_{s}(X_{s})\right\vert
\right\vert _{\mathscr{G}}^{2}ds<\infty \right) =1.  \label{c4bd}
\end{equation}%
In other words that the process X is a local semimartingale (the martingale
part of X is a local martingale and not a genuine martingale). The bound (\ref%
{c4bd}) also confirms the continuity of the martigale $X$ in $\mathscr{G}$
and that the solution of \eqref{eqnew:maingeneralnew} is indeed  strong.
\end{remark}

\begin{remark}
The common denominator of the above conditions is that 
\begin{equation*}
\mathbb{P}\left( \sup_{t\in \left[ 0,T\right] }\left\vert \left\vert X_{s}\right\vert
\right\vert _{\mathscr{F}_{0}}\right) =1
\end{equation*}%
in other words $X$ has paths in $L^{\infty }\left( [0,T],\mathscr{F}%
_{0}\right) ,$ $P$-almost surely, for any $T>0$.
\end{remark}

The proof of Theorem \ref{thmnew:maintheorem} will be presented in three parts, depending on the different values of $k$ and $\alpha_{i}$. More precisely:

\noindent\textbf{Proof of Theorem \ref{thmnew:maintheorem}}: 

In order to prove Theorem \ref{thmnew:maintheorem} we need the uniform
control and tightness for the approximating sequence. These are proven below
in Section \ref{sect:unifcontrol} and Section \ref{sect:tightness},
respectively. By Proposition \ref{tightness}, $(X^{d})_{d}$ is tight in $%
\mathscr{G}$, and therefore it has a weakly (in distribution)
convergent subsequence which we re-index by $(X^{d})_{d}${. By the Skorokhod
representation theorem, there exists a probability space $\left( \tilde{%
\Omega},\tilde{\mathcal{F}},\tilde{\mathbb{P}}\right) $ on which the
processes $\tilde{X}^{d}$ are defined and such that they have the same
distribution as $X^{d}$ and $\tilde{X}^{d}$ converges to a process $\tilde{X}
$, $\mathbb{\tilde{P}}$ - almost surely. The sequence }$(\tilde{X}%
_{t}^{d})_{d}$ inherits all the distributional properties from the sequence $%
(X^{d})_{d}.$ In particular the bounds in Propositions (\ref{prop:unifcontrol}) and (\ref{prop:unifcontrolinD}) are satisfied and
therefore $(\tilde{X}_{t}^{d})_{d}$ satisfies the controls (\ref{c1bd}), (\ref{c2bd}), (\ref{c3bd}), (\ref{c4bd}). By Fatou's lemma also the limiting
process {$\tilde{X}$ satisfies the same controls. }Moreover on {$\left( 
\tilde{\Omega},\tilde{\mathcal{F}},\tilde{\mathbb{P}}\right) $ there exists
a sequence of Bownian motions }$(\tilde{W}_{t}^{d})_{d}${\ such that} 
\begin{equation}
\tilde{X}_{t}^{d}=\tilde{X}_{0}^{d}+\displaystyle\int_{0}^{t}\tilde{A}%
_{s}^{d}(\tilde{X}_{s}^{d})ds+\displaystyle\int_{0}^{t}\tilde{B}_{s}^{d}(%
\tilde{X}_{s}^{d})d\tilde{W}_{s}^{d}  \label{eq:approx01}
\end{equation}%
more precisely 
\begin{equation}
\tilde{X}_{t}^{d}=\tilde{X}_{0}^{d}+\displaystyle\int_{0}^{t}\tilde{A}%
_{s}^{d}(\tilde{X}_{s}^{d})ds+\displaystyle\int_{0}^{t}\theta \displaystyle%
\Vert \tilde{X}_{s}^{d}\Vert _{\mathscr{F}_{i}}^{\alpha _{i}}\tilde{X}%
_{s}^{d}d\tilde{W}_{s}^{d}  \label{eq:approx02}
\end{equation}%
on $[0,T]$, with $i\in \{0,1\}$. 

{It also straightforward to deduce that the extended sequence }$(\tilde{X}%
^{d},\tilde{B}(\tilde{X}),\tilde{W}^{d})_{d}$ is tight in $\mathscr{G}\times 
$ $\mathscr{G}\times \mathbb{R}$ from which we can deduce by Theorem 4 in \cite{Jacubowski95} that there exists a subsequence,  for which we keep the same
index, such that
\begin{equation*}
\left( \tilde{X}^{d},\tilde{B}(\tilde{X}),\tilde{W}^{d},\int_{0}^{\cdot }%
\tilde{B}(\tilde{X}_{s})\tilde{W}_{s}^{d}\right) _{d}{\ }
\end{equation*}%
converges in distribution to 
\begin{equation*}
\left( \tilde{X},\tilde{B}(\tilde{X}),\tilde{W},\int_{0}^{\cdot }\tilde{B}(%
\tilde{X}_{s})\tilde{W}_{s}\right) _{d},
\end{equation*}%
where $\tilde{W}$ is a Brownian motion.  {By using one more time the
Skorokhod representation theorem, we can assume that all processes are
defined on the same probability space which we still denote by $\left( 
\tilde{\Omega},\tilde{\mathcal{F}},\tilde{\mathbb{P}}\right) $ and the
convergence is in }$\mathscr{G}\times $ $\mathscr{G}\times \mathbb{R\times }%
\mathscr{G}.$ Also, $\left( \tilde{X}^{d}\right) ${\ satisfies (\ref%
{eq:approx01}) and (\ref{eq:approx02}) on the new space.  }We will pass to
the limit in $\mathscr{G}$ each term of the equation \eqref{eq:approx01} for 
$\tilde{X}^{d}$. To show that the limiting process \~{X} {%
\eqref{eqnew:maingeneralnew}}. 

This, of course, holds immediately for the terms $\tilde{X}_{t}^{d}$ and $%
\tilde{X}_{0}^{d}$, we need to show that both $\int_{0}^{t}\tilde{A}_{s}^{d}(%
\tilde{X}_{s}^{d})ds$ and $\int_{0}^{t}\tilde{B}_{s}^{d}(\tilde{X}%
_{s}^{d})dW_{s}^{d}$ converge to $\int_{0}^{t}\tilde{A}(\tilde{X})ds$ and $%
\int_{0}^{t}\tilde{B}(\tilde{X})dW$, respectively, \textbf{in $\mathscr{G}$}.
 
For the first integral we will use that
\begin{equation}
\begin{aligned}
\|\tilde{A}^d(\tilde{X}^d)-\tilde{A}(\tilde{X})\|_{\mathscr{G}} \leq \|\tilde{A}^d(\tilde{X}^d) - \tilde{A}^d(\tilde{X})\|_{\mathscr{G}} + \|\tilde{A}^d(\tilde{X}) - \tilde{A}(\tilde{X})\|_{\mathscr{G}}. 
\end{aligned}
\end{equation}
The control of the terms above depends on the initial assumptions, which are different in each of the three cases. We need, therefore, to treat each of these cases separately. 

\noindent\textbf{Case I:}
We have from Assumption $(A1)$ that
\begin{equation}
\begin{aligned}\label{conv1}
\|\tilde{A}^d(\tilde{X}^d) - \tilde{A}^d(\tilde{X})\|_{\mathscr{G}} \leq C(1 + \|\tilde{X}^d\|_{\mathscr{F}_1} +  \|\tilde{X}\|_{\mathscr{F}_1})\|\tilde{X}^d - \tilde{X}\|_{\mathscr{F}_0} %\xrightarrow{d\rightarrow \infty} 0.
\end{aligned}
\end{equation}
But we know from Proposition \ref{tightness} and Proposition \ref{prop:unifcontrol} that $(\tilde{X}^d)_d$ is tight in $\mathscr{G}$ and there is a uniform control for $\displaystyle\int_0^t\|\tilde{X}_s^d\|_{\mathscr{F}_1}ds$. Then 
\begin{equation}
\|\tilde{X}^d - \tilde{X}\|_{\mathscr{F}_0} \xrightarrow{d\rightarrow \infty} 0. 
\end{equation}
It also holds that 
\begin{equation}
\begin{aligned}
\|\tilde{A}^d(\tilde{X}) - \tilde{A}(\tilde{X})\|_{\mathscr{G}} \xrightarrow{d\rightarrow \infty} 0
\end{aligned}
\end{equation}
since $\tilde{A}^d$ is the projection of $\tilde{A}$ through the operator $T_d$. 
More precisely, in the first case we have that%
\begin{equation*}
\begin{aligned}
\int_{0}^{t}\Vert \tilde{A}^{d}(\tilde{X}_{s})-\tilde{A}(\tilde{X}_{s})\Vert _{\mathscr{G}}ds &
\leq \int_{0}^{t}\Vert \tilde{A}^{d}(\tilde{X}_{s})\Vert _{\mathscr{G}}+\Vert
\tilde{A}(\tilde{X}_{s})\Vert _{\mathscr{G}}ds \leq 2\int_{0}^{t}\Vert \tilde{A}(\tilde{X}_{s})\Vert _{\mathscr{G}}ds \\
&\leq 2C_1\left( \int_{0}^{t}\left\Vert \tilde{X}_{s}\right\Vert _{\mathscr{F}_{0}}^{\gamma ^1}ds+C_2\int_{0}^{t}\left\Vert \tilde{X}_{s}\right\Vert_{\mathscr{F}_{1}}^{2}ds+1\right) 
\end{aligned}
\end{equation*}
which holds due to Assumption $(A1)$ (with $\gamma^2=2$) for $\tilde{A}^d$ and Fatou lemma. 
Then since 
\begin{equation*}
\lim_{d\rightarrow \infty }\Vert \tilde{A}^{d}(\tilde{X}_{s})-\tilde{A}(\tilde{X}_{s})\Vert_{\mathscr{G}}=0
\end{equation*}
it follows by the dominated convergence theorem that 
\begin{equation*}
\displaystyle\lim_{d\rightarrow \infty }\int_{0}^{t}\|\tilde{A}^{d}(\tilde{X}_{s})-\tilde{A}(\tilde{X}_{s})\| _{\mathscr{G}}ds=0
\end{equation*}%
$\mathbb{P}$-almost surely and therefore also in probability. 
Next we also have that 
\begin{equation*}
\begin{aligned}
\| \tilde{A}^{d}(\tilde{X}_{s}^{d})-\tilde{A}^{d}(\tilde{X}_{s})\Vert _{\mathscr{G}} &\leq \|
\tilde{A}(\tilde{X}_{s}^{d})-\tilde{A}(\tilde{X}_{s})\Vert _{\mathscr{G}} \leq C\left( 1+\|\tilde{X}_{s}^{d}\|_{\mathscr{F}_{1}} + \|
\tilde{X}_{s}\Vert _{\mathscr{F}_{1}}\right) \| \tilde{X}_{s}^{d}-\tilde{X}_{s}\|_{\mathscr{F}_{0}}
\end{aligned}
\end{equation*}
based on the same assumption $(A1)$ and Fatou lemma.
Hence
\begin{equation*}
\left(\int_{0}^{t}\|A^{d}(X_{s}^{d})-A^{d}(X_{s})\|_{\mathscr{G}%
}ds\right) ^{2}\leq 4C^2 \int_{0}^{t}\left( 1+ \|X_{s}^{d}\|_{\mathscr{F}_{1}}^2+ \|
X_{s}\Vert _{\mathscr{F}_{1}}^2\right)
ds\int_{0}^{t}\| X_{s}^{d}-X_{s}\|_{\mathscr{F}_{0}}^{2}ds
\end{equation*}%
and therefore 
\begin{equation*}
\begin{aligned}
\lim_{d\rightarrow \infty }\left( \int_{0}^{t}\| A^{d}(X_{s}^{d})-A^{d}(X_{s})\|_{\mathscr{G}}ds\right) ^{2} &\leq 
\displaystyle\limsup_{d\rightarrow\infty}\displaystyle\int_0^t\left( 1+ \|X_{s}^{d}\|_{\mathscr{F}_{1}}^2+ \|
X_{s}\Vert _{\mathscr{F}_{1}}^2\right)ds
\lim_{d\rightarrow \infty }\int_{0}^{t}\|
X_{s}^{d}-X_{s}\|_{\mathscr{F}_{0}}^{2}ds \\& = 0 \\
%&\leq 2\int_{0}^{t}\left(\|X_{s}\|_{\mathscr{F}_{0}}^{2\gamma}\right) ds\times 0.
\end{aligned}
\end{equation*}
Note that the first term is uniformly bounded in $d$ from Proposition \ref{prop:unifcontrol} hence the limit is indeed 0.

\noindent\textbf{Case II:} This follows using exactly the same arguments, since condition \eqref{cond:convA1} is the same. 
We have, more precisely,
\begin{equation*}
\begin{aligned}
\int_{0}^{t}\Vert \tilde{A}^{d}(\tilde{X}_{s})-\tilde{A}(\tilde{X}_{s})\Vert _{\mathscr{G}}ds 
&\leq 2C_1 \int_{0}^{t}\left\Vert \tilde{X}_{s}\right\Vert _{\mathscr{F}_{0}}^{\gamma^1}ds
\end{aligned}
\end{equation*}
and
\begin{equation*}
\begin{aligned}
\| \tilde{A}^{d}(\tilde{X}_{s}^{d})-\tilde{A}^{d}(\tilde{X}_{s})\Vert _{\mathscr{G}} &\leq \|
\tilde{A}(\tilde{X}_{s}^{d})-\tilde{A}(\tilde{X}_{s})\Vert_{\mathscr{G}} \leq C\left( 1+\|\tilde{X}_{s}^{d}\|_{\mathscr{F}_{1}} + \|
\tilde{X}_{s}\Vert _{\mathscr{F}_{1}}\right) \| \tilde{X}_{s}^{d}-\tilde{X}_{s}\|_{\mathscr{F}_{0}}
\end{aligned}
\end{equation*}
and the limit when $d\rightarrow\infty$ is zero as before.

\noindent\textbf{Case III:}
From Proposition \ref{tightness} and Proposition \ref{prop:unifcontrol},
\begin{equation}
\|\tilde{X}^d - \tilde{X}\|_{\mathscr{F}_1} \xrightarrow{d\rightarrow \infty} 0. 
\end{equation}
Also
\begin{equation}
\begin{aligned}
\|\tilde{A}^d(\tilde{X}) - \tilde{A}(\tilde{X})\|_{\mathscr{G}} \xrightarrow{d\rightarrow \infty} 0
\end{aligned}
\end{equation}
by the same arguments as before. 
Then
\begin{equation*}
\begin{aligned}
\int_{0}^{t}\Vert \tilde{A}^{d}(\tilde{X}_{s})-\tilde{A}(\tilde{X}_{s})\Vert _{\mathscr{G}}ds 
&\leq &2C_2 \int_{0}^{t}\left\Vert \tilde{X}_{s}\right\Vert _{\mathscr{F}_{1}}^{\gamma^2}ds
\end{aligned}
\end{equation*}
Then by the dominated convergence theorem 
\begin{equation*}
\displaystyle\lim_{d\rightarrow \infty }\int_{0}^{t}\|\tilde{A}^{d}(\tilde{X}_{s})-\tilde{A}(\tilde{X}_{s})\| _{\mathscr{G}}ds=0
\end{equation*}%
$\mathbb{P}$-almost surely and therefore also in probability. 
Next we have that 
\begin{equation}
\begin{aligned}
\| \tilde{A}^{d}(\tilde{X}_{s}^{d})-\tilde{A}^{d}(\tilde{X}_{s})\Vert _{\mathscr{G}} &\leq \|
\tilde{A}(\tilde{X}_{s}^{d})-\tilde{A}(\tilde{X}_{s})\Vert _{\mathscr{G}} \leq C\left( 1+\|\tilde{X}_{s}^{d}\|_{\mathscr{F}_{1}} + \|
\tilde{X}_{s}\Vert _{\mathscr{F}_{1}}\right) \| \tilde{X}_{s}^{d}-\tilde{X}_{s}\|_{\mathscr{F}_{1}}
\end{aligned}
\end{equation}
based on assumption $(A3)$ and Fatou lemma.
Hence
\begin{equation*}
\left(\int_{0}^{t}\|A^{d}(X_{s}^{d})-A^{d}(X_{s})\|_{\mathscr{G}%
}ds\right) ^{2}\leq 4C^2 \int_{0}^{t}\left( 1+ \|X_{s}^{d}\|_{\mathscr{F}_{1}}^2+ \|
X_{s}\Vert _{\mathscr{F}_{1}}^2\right)
ds\int_{0}^{t}\| X_{s}^{d}-X_{s}\|_{\mathscr{F}_{1}}^{2}ds
\end{equation*}%
and 
\begin{equation*}
\begin{aligned}
\lim_{d\rightarrow \infty }\left( \int_{0}^{t}\| A^{d}(X_{s}^{d})-A^{d}(X_{s})\|_{\mathscr{G}}ds\right) ^{2}  = 0 \\
%&\leq 2\int_{0}^{t}\left(\|X_{s}\|_{\mathscr{F}_{0}}^{2\gamma}\right) ds\times 0.
\end{aligned}
\end{equation*}
as before. 

We now justify the convergence of the stochastic integral. Let 
\begin{equation}
\begin{aligned}
a_d(t) := \displaystyle\int_0^t \tilde{B}_s(\tilde{X}_s)d\tilde{W}_s^d - \displaystyle\int_0^t \tilde{B}_s(\tilde{X}_s)d\tilde{W}_s^d
\end{aligned}
\end{equation}
\begin{equation}
\begin{aligned}
b_d(t) := \displaystyle\int_0^t \tilde{B}_s^d(\tilde{X}_s^d)d\tilde{W}_s^d - \displaystyle\int_0^t \tilde{B}_s(\tilde{X}_s)d\tilde{W}_s
\end{aligned}
\end{equation}
We need to show that $a_d(t) + b_d(t) \xrightarrow{d\rightarrow\infty} 0$ in probability.
We already know that  $a_d(t) \xrightarrow{d\rightarrow\infty} 0$ almost surely in $\mathscr{G}$
(hence in probability). By Lemma \ref{lemma:conv} below 
\begin{equation}
\displaystyle\int_0^t \|\tilde{B}^d(\tilde{X^d}) - \tilde{B}(\tilde{X})\|_{\mathscr{G}}^2ds \xrightarrow{d\rightarrow \infty}0
\end{equation}
in probability. Due to this and using Proposition \ref{gawareckimandrekar} we can write
\begin{equation}
\mathbb{P}\left( \displaystyle\sup_{s\in[0,t]}\|b_d(s)\|_{\mathscr{G}} > \epsilon\right) \leq \frac{\eta}{\epsilon^2} + \mathbb{P}\left(\displaystyle\int_0^t \|\tilde{B}_s^d(\tilde{X_s^d}) - \tilde{B}_s(\tilde{X}_s)\|_{\mathscr{G}}^2 > \eta \right).
\end{equation}
Then we can choose $\eta:=\frac{\delta}{2\epsilon^2}$ and $d$ sufficiently large such that
\begin{equation}
\mathbb{P}\left(\displaystyle\int_0^t \|\tilde{B}_s^d(\tilde{X_s^d}) - \tilde{B}_s(\tilde{X}_s)\|_{\mathscr{G}}^2 > \frac{\delta}{2\epsilon^2} \right) < \frac{\delta}{2}
\end{equation}
which gives that also $b_d(t) \xrightarrow{d\rightarrow\infty} 0$ in probability.

\vspace{10mm}

\begin{lemma}\label{lemma:conv}
We have
\begin{equation}
\displaystyle\int_0^t \|\tilde{B}_s^d(\tilde{X_s^d}) - \tilde{B}_s(\tilde{X}_s)\|_{\mathscr{G}}^2ds \xrightarrow{d\rightarrow \infty}0
\end{equation}
in probability.
\end{lemma}
\begin{proof}
\noindent\textbf{Case I}: We can write
\begin{equation*}
\begin{aligned}
\|\tilde{B}_s^d(\tilde{X_s^d}) - \tilde{B}_s(\tilde{X}_s)\|_{\mathscr{G}} & \leq \|\tilde{X}_s^d\|_{\mathscr{F}_0}\left(\|\tilde{X}_s^d - \tilde{X}_s\|_{\mathscr{G}} \right) + \left(\|X_s^d\|_{\mathscr{F}_0} - \|\tilde{X}_s\|_{\mathscr{F}_0}\right)\|\tilde{X}_s\|_{\mathscr{G}} \\
& \leq \displaystyle\sup_{s}\|\tilde{X}_s^d\|_{\mathscr{F}_0}\left(\|\tilde{X}_s^d - \tilde{X}_s\|_{\mathscr{G}} \right) + \displaystyle\sup_{s}\|\tilde{X}_s\|_{\mathscr{G}}\left(\|\tilde{X}_s^d\|_{\mathscr{F}_0} - \|\tilde{X}_s\|_{\mathscr{F}_0}\right)
\end{aligned}
\end{equation*}
We will prove that each of these two terms converges to zero in probability. 
Let 
\begin{equation*}
c_d(t):= \displaystyle\sup_{s\in[0,t]}\|\tilde{X}_s^d\|_{\mathscr{F}_0}\displaystyle\int_0^t \|\tilde{X}_s^d - \tilde{X}_s\|_{\mathscr{G}} ds
\end{equation*}

\begin{equation*}
e_d(t):= \displaystyle\sup_{s\in[0,t]}\|\tilde{X}_s\|_{\mathscr{G}}\displaystyle\int_0^t\|\tilde{X}_s^d -\tilde{X}_s\|_{\mathscr{F}_0}ds.
\end{equation*}
Note that since $X_s^d \rightarrow X_s$ in $\mathscr{G}$ and 
\begin{equation*}
\|\tilde{X}_s^d-\tilde{X}_s\|_{\mathscr{G}} \leq C\displaystyle\sup_{s}\|\tilde{X}_s\|_{\mathscr{G}} \leq C\displaystyle\sup_{s}\|\tilde{X}_s\|_{\mathscr{F}_0}
\end{equation*}
which is bounded in probability by Proposition \ref{prop:unifcontrol} and the Fatou lemma, we have
\begin{equation*}\label{eqlemma:conv}
\displaystyle\lim_{d\rightarrow\infty} \displaystyle\int_0^t\|\tilde{X}_s^d-\tilde{X}_s\|_{\mathscr{G}} ds = 0
\end{equation*}
in probability.
We can choose $R$ such that 
\begin{equation}
\mathbb{P}\left(\displaystyle\sup_{s\in[0,T]}\|\tilde{X}_s^d\|_{\mathscr{F}_0} >R\right) < \frac{\delta}{2}
\end{equation}
and $d$ such that
\begin{equation}
\mathbb{P}\left(\displaystyle\int_0^t\|\tilde{X}_s^d-\tilde{X}_s\|_{\mathscr{G}}ds > \frac{\epsilon}{R} \right) < \frac{\delta}{2}.
\end{equation}
It follows that
\begin{equation*}
\begin{aligned}
\mathbb{P}\left( c_d(t) \geq \epsilon \right) & = \mathbb{P}\left( c_d(t) \geq \epsilon, \displaystyle\sup_{s\in[0,T]}\|\tilde{X}_s^d\|_{\mathscr{F}_0} >R\right) + \mathbb{P}\left( c_d(t) \geq \epsilon, \displaystyle\sup_{s\in[0,T]}\|\tilde{X}_s^d\|_{\mathscr{F}_0} < R   \right)\\
& \leq \frac{\delta}{2} + \mathbb{P}\left( c_d(t) \geq \epsilon, \displaystyle\sup_{s\in[0,T]}\|\tilde{X}_s^d\|_{\mathscr{F}_0} >R, \displaystyle\int_0^t\|\tilde{X}_s^d-\tilde{X}_s\|_{\mathscr{G}}ds > \frac{\epsilon}{R}\right) \\
& + \mathbb{P}\left( c_d(t) \geq \epsilon, \displaystyle\sup_{s\in[0,T]}\|\tilde{X}_s^d\|_{\mathscr{F}_0} >R, \displaystyle\int_0^t\|\tilde{X}_s^d-\tilde{X}_s\|_{\mathscr{G}}ds < \frac{\epsilon}{R}\right) \\
& \leq \frac{\delta}{2} + \mathbb{P}\left(\displaystyle\int_0^t\|\tilde{X}_s^d-\tilde{X}_s\|_{\mathscr{G}}ds > \frac{\epsilon}{R} \right) \\
& \leq \frac{\delta}{2} + \frac{\delta}{2} = \delta.
\end{aligned}
\end{equation*}
For showing convergence (in probability) of $e_d(t)$ we proceed in a similar way, using the fact that $\displaystyle\sup_{s}\|\tilde{X}_s\|_{\mathscr{G}}$ is bounded in probability and that
\begin{equation*}
\begin{aligned}
\displaystyle\int_0^t\|\tilde{X}_s^d-\tilde{X}_s\|_{\mathscr{F}_0}ds & \leq \displaystyle\int_0^t\|\tilde{X}_s^d-\tilde{X}_s\|_{\mathscr{F}_1}^{\theta^{\star}}\|\tilde{X}_s^d-\tilde{X}_s\|_{\mathscr{G}}^{1-\theta^{\star}}ds \\
& \leq \left(\displaystyle\int_0^t\|\tilde{X}_s^d-\tilde{X}_s\|_{\mathscr{F}_1}ds \right)^{1/2} \left( \displaystyle\int_0^t \|\tilde{X}_s^d-\tilde{X}_s\|_{\mathscr{G}}^{1-\theta^{\star}}\right)^{1/2}
\end{aligned}
\end{equation*}
and the last term converges to zero in probability, while the integral of the $\mathscr{F}_1$ norm of the difference is bounded in probability. 
The control in \eqref{lemma:conv} holds also in Case II and Case III, by the same arguments. 
\end{proof}

\subsection{Uniform control}\label{sect:unifcontrol}
\begin{proposition}\label{prop:unifcontrol}
The approximating sequence $(X_t^d)_d$ is $\mathbb{P}$-a.s. uniformly controlled in $\mathscr{F}_0$ in the uniform norm, and in $\mathscr{F}$ in $L^2$ sense. That is for any $\epsilon_1,\epsilon_2 >0$ there exists $\mathcal{K}_1,\mathcal{K}_2 >0$ such that
\begin{equation}\label{eq:unifcontrol}
\displaystyle\sup_d \mathbb{P}\left( \displaystyle\sup_t \|X_t^d\|_{\mathscr{F}_0}^2 \geq \mathcal{K}_1\right) \leq \epsilon_1
\end{equation}
and 
\begin{equation}\label{eq:unifcontrolint}
\displaystyle\sup_d \mathbb{P}\left( \displaystyle\int_0^T \|X_s^d\|_{\mathscr{F}_1}^2ds \geq \mathcal{K}_2\right) \leq \epsilon_2.
\end{equation}
\end{proposition}

\begin{proof}
We first prove \eqref{eq:unifcontrol} and then deduce \eqref{eq:unifcontrolint} from it, when the space $\mathscr{F}_1$ exists. 
In order to prove \eqref{eq:unifcontrol}, we first show that for any $\tilde{\epsilon}>0$ there exists $\tilde{\mathcal{K}}$ such that 
\begin{equation}
\displaystyle\sup_d \mathbb{P}\left( \displaystyle\sup_t  \log(C+\|X_t^d\|_{\mathscr{F}_0}^2 )\geq \tilde{\mathcal{K}}\right) \leq \tilde{\epsilon}
\end{equation}
and then the conclusion will follow. 
By the It\^{o} formula we have 
\begin{equation}
\begin{aligned}
d\|X_t^d\|_{\mathscr{F}_0}^2 & = \tilde{A}(X_t^d)dt + \tilde{B}(X_t^d)dW_t^d
\end{aligned}
\end{equation}
where
\begin{equation}
\tilde{A}(X_t^d):= \left(2{}_\mathscr{D}\langle X_t^d, A^d(X_t^d)\rangle_{\mathscr{G}} + \|B_t(X_t^d)\|_{L_2(\mathscr{U}, \mathscr{F}_0)}^2\right)
\end{equation}
and 
\begin{equation}
\tilde{B}(X_t^d)dW_t^d:= \langle X_t^d, B_t(X_t^d)dW_t\rangle_{\mathscr{F}_0}
\end{equation}
Then
\begin{equation}\label{logeqn}
\begin{aligned}
d\log  \left( C+ \|X_t^d\|_{\mathscr{F}_0}^2 \right) &= \frac{1}{C+ \|X_t^d\|_{\mathscr{F}_0}^2}\left(\tilde{A}_t^d(X_t^d)dt + \tilde{B}_t^d(X_t^d)dW_t^d \right) - \frac{1}{2(C+ \|X_t^d\|_{\mathscr{F}_0}^2)^2}\tilde{B}_t
^d(X_t^d)^2dt \\
& = \frac{1}{C+ \|X_t^d\|_{\mathscr{F}_0}^2}\tilde{A}_t^d(X_t^d)dt - \frac{1-\epsilon}{2(C+ \|X_t^d\|_{\mathscr{F}_0}^2)^2}\tilde{B}_t
^d(X_t^d)^2dt \\
& + \frac{1}{C+ \|X_t^d\|_{\mathscr{F}_0}^2}\tilde{B}_t^d(X_t^d)dW_t^d  - \frac{\epsilon}{2(C+ \|X_t^d\|_{\mathscr{F}_0}^2)^2}\tilde{B}_t
^d(X_t^d)^2dt.
\end{aligned}
\end{equation}
If we choose
\begin{equation}
M_t:=\displaystyle\int_0^t\tilde{B}^d(X_s^d)dW_s
\end{equation}
then
\begin{equation}
\langle M \rangle_t = \displaystyle\int_0^t\tilde{B}^d(X_s)^2 ds
\end{equation}
As suggested above, we will now need to split the proof in three distinct cases, depending on the specific form of the diffusion term $B$:

\vspace{3mm}
\noindent\textbf{Case I:} the diffusion term is given by
\begin{equation}\label{Bcompressviscousgen}
B_t^d(X_t^d) = \theta\|X_t^d\|_{\mathscr{F}_0}^{\alpha_0}X_t^d.
\end{equation}
In this case
\begin{equation}
\begin{aligned}
\|B_t^d(X_t^d)\|_{L_2(\mathscr{U}, \mathscr{F}_0)}^2 
& = \theta^2\|X_t^d\|_{\mathscr{F}_0}^{2\alpha_0+2} 
\end{aligned}
\end{equation}
\begin{equation}
\tilde{B}_t^d(X_t^d)^2:=(2\langle X_t^d, B_t^d(X_t^d)\rangle)^2 = 4\theta^2 \|X_t^d\|_{\mathscr{F}_0}^{2\alpha_0+4}
\end{equation}
and
\begin{equation}
\begin{aligned}
\tilde{A}^d(X^d) &=2{}_\mathscr{D}\langle X_t^d, A^d(X_t^d)\rangle_{\mathscr{G}} + \|B_t(X_t^d)\|_{L_2(\mathscr{U}, \mathscr{F}_0)}^2 \\
%& = 2{}_\mathscr{D}\langle X_t^d, A^d(X_t^d)\rangle_{\mathscr{G}} + \theta^2\|X_t^d\|_{\mathscr{F}_0}^{2\alpha_0+2} \\
& = 2\langle X_t^d, A^d(X_t^d)\rangle_{\mathscr{F}_0} + \theta^2\|X_t^d\|_{\mathscr{F}_0}^{2\alpha_0+2}\\
& \leq 2C_1\|X_t^d\|_{\mathscr{F}_0}^{\gamma_1} - 2C_2\|X_t^d\|_{\mathscr{F}_1}^{\gamma_2}+ \theta^2\|X_t^d\|_{\mathscr{F}_0}^{2\alpha_0+2}.
\end{aligned}
\end{equation}
due to conditions \eqref{condscproduct} and \eqref{cond:compressviscous}.
Let us define the semimartingale 
\begin{equation}
dM_t := 2\theta_1\langle X_t^d, \|X_t^d\|_{\mathscr{F}_0}^{\alpha_1}X_t^d\rangle dW_t^d.
\end{equation}
Then
\begin{equation}\label{viscousineqgen}
\begin{aligned}
d\|X_t^d\|_{\mathscr{F}_0}^2 + \tilde{C}_2\|X_t^d\|_{\mathscr{F}_1}^{\gamma_2}dt \leq G \left( \|X_t^d\|_{\mathscr{F}_0}\right)dt + dM_t
\end{aligned}
\end{equation}
where $\tilde{C}_2 >0$ and it depends on $C_2$.\footnote{Note that $\tilde{C}_2$ can be equal to $2C_2, C_2$, or something else, depending on the model. See Section \ref{sect:applications} for more explicit details.}
\begin{equation}
G \left( \|X_t^d\|_{\mathscr{F}_0}\right) := \tilde{C}_1\|X_t^d\|_{\mathscr{F}_0}^{\gamma_1} + \theta^2\|X_t^d\|_{\mathscr{F}_0}^{2\alpha_0+2}.
\end{equation}
But then 
\begin{equation}
\begin{aligned}
d\|X_t^d\|_{\mathscr{F}_0}^2  \leq G(\|X_t^d\|_{\mathscr{F}_0})dt + \tilde{B}(X_t^d)dW_t
\end{aligned}
\end{equation}
and for $C>0$ and 
\begin{equation}
\tilde{C}(X_t^d) := C+ \|X_t^d\|_{\mathscr{F}_0}^2
\end{equation}
we can write 
\begin{equation}
\begin{aligned}
d &\log(C+\|X_t^d\|_{\mathscr{F}_0}^2)  \leq \frac{1}{\tilde{C}(X_t^d)^2}\left(\tilde{C}(X_t^d)G(\|X_t^d\|_{\mathscr{F}_0}) - \frac{1-\epsilon}{2}\tilde{B}^d(X_t^d)^2\right)dt + dM_t - \frac{\epsilon}{2}d\langle M\rangle_t\\
& = \frac{1}{\tilde{C}(X_t^d)^2}\left(\tilde{C}(X_t^d)(\tilde{C}_1\|X_t^d\|_{\mathscr{F}_0}^{\gamma_1}+\|B^d(X_t^d)\|_{L_2(\mathscr{U}, \mathscr{F}_0)}^2 ) - \frac{1-\epsilon}{2}\tilde{B}^d(X_t^d)^2\right)dt + dM_t - \frac{\epsilon}{2}d\langle M\rangle_t \\
& = \frac{1}{(C+ \|X_t^d\|_{\mathscr{F}_0}^2)^2} \left( (C+ \|X_t^d\|_{\mathscr{F}_0}^2)(\tilde{C}_1\|X_t^d\|_{\mathscr{F}_0}^{\gamma_1} + \theta^2\|X_t^d\|_{\mathscr{F}_0}^{2\alpha_0+2})- \frac{1-\epsilon}{2}4\theta^2\|X_t^d\|_{\mathscr{F}_0}^{2\alpha_0+4}\right)dt \\
& + dM_t - \frac{\epsilon}{2}d\langle M\rangle_t \\
& = \frac{1}{(C+ \|X_t^d\|_{\mathscr{F}_0}^2)^2} \left(\tilde{C}_{11}\|X_t^d\|_{\mathscr{F}_0}^{\gamma_1} + C\theta^2 \|X_t^d\|_{\mathscr{F}_0}^{2\alpha_0+2} + \tilde{C}_1\|X_t^d\|_{\mathscr{F}_0}^{\gamma_1+2} +  \theta^2\|X_t^d\|_{\mathscr{F}_0}^{2\alpha_0+4} \right. \\
& \left. - \frac{1-\epsilon}{2}4\theta^2\|X_t^d\|_{\mathscr{F}_0}^{2\alpha_0+4}\right)dt + dM_t - \frac{\epsilon}{2}d\langle M\rangle_t \\
& = \frac{1}{(C+ \|X_t^d\|_{\mathscr{F}_0}^2)^2} \left(\tilde{C}_1\|X_t^d\|_{\mathscr{F}_0}^{\gamma_1+2} + \tilde{C}_{11}\|X_t^d\|_{\mathscr{F}_0}^{\gamma_1} + C\theta^2 \|X_t^d\|_{\mathscr{F}_0}^{2\alpha_0+2}  + (2\epsilon-1)\theta^2\|X_t^d\|_{\mathscr{F}_0}^{2\alpha_0+4}\right)dt\\
& + dM_t - \frac{\epsilon}{2}d\langle M\rangle_t.
\end{aligned}
\end{equation}
We want $2\epsilon-1 <0$ and for this reason we choose $\epsilon < \frac{1}{2}$. We know by assumption also that $2\alpha_0>\gamma_1$.
Therefore we control any possible blow-up of the drift term, since there exists a positive constant $R$ such that
\begin{equation}\label{eq:330}
\tilde{C}_1\|X_t^d\|_{\mathscr{F}_0}^{\gamma_1+2} + \tilde{C}_{11}\|X_t^d\|_{\mathscr{F}_0}^{\gamma_1} + C\theta^2 \|X_t^d\|_{\mathscr{F}_0}^{2\alpha_0+2}  + (2\epsilon-1)\theta^2\|X_t^d\|_{\mathscr{F}_0}^{2\alpha_0+4} \leq R.
\end{equation}
Then, as the term on the left hand side of \eqref{eq:330} has negative leading coefficient, we eventually obtain
\begin{equation}
d\log (C+\|X_t^d\|_{\mathscr{F}_0}^2 )\leq k_0 + kt + E(\epsilon)
\end{equation}
with
\begin{equation}
E(\epsilon) := \sup_{t\ge 0}\left(M_t - \frac{\epsilon}{2}\langle M\rangle_t\right).
\end{equation}
Therefore
\begin{equation}
\displaystyle\sup_{t\in[0,T]}\|X_t^d\|_{\mathscr{F}_0}^2 \leq \tilde{E}(\epsilon)
\end{equation}
where
\begin{equation}\label{ed}
\tilde{E}(\epsilon) = \exp (k_0+kT+E(\epsilon))-C.
\end{equation}
Using Exercise 3.16 from Revuz \& Yor we obtain
\begin{equation}
\mathbb{P}\left( \displaystyle\sup_{t\in[0,T]}\|X_t^d\|_{\mathscr{F}_0}^2 \geq \beta_1\right) < \delta_1.
\end{equation}
Note that the random variable on the right hand side of \eqref{ed} is finite, but not integrable. It follows that 
\begin{equation}
\mathbb{P}(\sup_{t\in [0,T]} (C+\|X_t^d\|_{\mathscr{F}_0}) > K)
\le 
\mathbb{P}(E(\epsilon) > \log K - k_0+kT)
\end{equation}
Note also that the quantity on the right hand side 
of the above inequality is independent of the dimension $d$, (as $E(\epsilon)$ is $d$-independent and has an exponential distribution), hence 
\begin{equation}
\displaystyle\lim_{K\rightarrow \infty} \sup_{d} \mathbb{P}(\sup_{t\in [0,T]} {\|X_t^d\|_{\mathscr{F}_0}} > K)
\le 
\displaystyle\lim_{K\rightarrow \infty} \mathbb{P}(E(\epsilon) > \log K - k_0+kT)=0.
\end{equation}
Now we can proceed with controlling $\tilde{C}_2\int_0^t \|X_s^d\|_{\mathscr{F}_1}^{\gamma_2}ds$. From \eqref{viscousineqgen} one has
\begin{equation}
\begin{aligned}
\tilde{C}_2\displaystyle\int_0^t \|X_s^d|_{\mathscr{F}_1}^{\gamma_2} \leq k_0 - \|X_t^d\|_{\mathscr{F}_0}^2 + \displaystyle\int_0^t G(\|X_s^d\|_{\mathscr{F}_0})ds + \displaystyle\int_0^t \tilde{B}_s^d(X_s^d)dW_s.
\end{aligned}
\end{equation}
Let 
\begin{equation}
    \begin{aligned}
    H_s(\|X_t^d\|_{\mathscr{F}_0}) := k_0 - \|X_t^d\|_{\mathscr{F}_0}^2 + \displaystyle\int_0^t G(\|X_s^d\|_{\mathscr{F}_0})ds
    \end{aligned}
\end{equation}
and
\begin{equation}
\tilde{\tilde{B}}_t := \displaystyle\int_0^t \tilde{B}_s^d(X_s^d)dW_s
\end{equation}
Then 
\begin{equation}
\tilde{C}_2\displaystyle\int_0^t \|X_s^d\|_{\mathscr{F}_1}^2 \leq 
H_s(\|X_t^d\|_{\mathscr{F}_0})+\frac{\epsilon}{2}\langle \tilde{\tilde{B}}\rangle_t+\tilde E(\epsilon)
\le C(T,\epsilon)(\sup_{t\in[0,T]}\|X_t^d\|_{\mathscr{F}_0}^{q'})+\tilde E(\epsilon)
\end{equation}
with $
\tilde E(\epsilon) := \sup_{t\ge 0}(\tilde{\tilde{B}}_t - \frac{\epsilon}{2}\langle \tilde{\tilde{B}}\rangle_t)$ and $q'$ suitably chosen.
%(\OL{We don't yet have control on $\displaystyle\sup_{t}\|X_t^d\|_{\mathscr{F}_0}^{q'}$ - I guess we can conclude it directly (as in the paper with E.), since $q'=\gamma_1$ and we've chosen $\gamma_1 \geq 2$.)}
This can now be handled as before, which will eventually give
\begin{equation}
\mathbb{P}\left( \displaystyle\int_0^t\|X_t^d\|_{\mathscr{F}_1}^{\gamma_2}ds \geq \beta_2\right) < \delta_2
\end{equation}
with $\tilde{C}_2, \gamma_2 > 0$.

\vspace{3mm}
\noindent\textbf{Case II:} the diffusion term is given by
\begin{equation}\label{Bcompressinviscidgen}
B_t^d(X_t^d):= \theta\|X_t^d\|_{\mathscr{F}_1}^{\alpha_1}X_t^d. 
\end{equation}
Then
\begin{equation}
\begin{aligned}
\|B_t^d(X_t^d)\|_{L_2(\mathscr{U}, \mathscr{F}_0)}^2 
& = \theta^2\|X_t^d\|_{\mathscr{F}_1}^{2\alpha_1}\|X_t^d\|_{\mathscr{F}_0}^{2} 
\end{aligned}
\end{equation}
\begin{equation}
\tilde{B}_t^d(X_t^d)^2 = 4\theta^2\|X_t^d\|_{\mathscr{F}_1}^{2\alpha_1}\|X_t^d\|_{\mathscr{F}_0}^{4}. 
\end{equation}
and
\begin{equation}
\begin{aligned}
\tilde{A}^d(X^d) &=2{}_\mathscr{D}\langle X_t^d, A^d(X_t^d)\rangle_{\mathscr{G}} + \|B_t(X_t^d)\|_{L_2(\mathscr{U}, \mathscr{F}_0)}^2 \\
%& = 2{}_\mathscr{D}\langle X_t^d, A^d(X_t^d)\rangle_{\mathscr{G}} + \theta^2\|X_t^d\|_{\mathscr{F}_1}^{2\alpha_1}\|X_t^d\|_{\mathscr{F}_0}^{2} \\
& = 2\langle X_t^d, A^d(X_t^d)\rangle_{\mathscr{F}_0} + \theta^2\|X_t^d\|_{\mathscr{F}_1}^{2\alpha_1}\|X_t^d\|_{\mathscr{F}_0}^{2}\\
& \leq 2C_1\|X_t^d\|_{\mathscr{F}_0}^{\gamma_1} - 2C_2\|X_t^d\|_{\mathscr{F}_1}^{\gamma_2}+ \theta^2\|X_t^d\|_{\mathscr{F}_1}^{2\alpha_1}\|X_t^d\|_{\mathscr{F}_0}^{2}.
\end{aligned}
\end{equation}
We have, as before, 
\begin{equation}\label{aneqngen}
\begin{aligned}
d\|X_t^d\|_{\mathscr{F}_0}^2 &= \left(2{}_\mathscr{D}\langle X_t^d, A^d(X_t^d)\rangle_{\mathscr{G}} + \|B_t^d(X_t^d)\|_{L_2(\mathscr{U}, \mathscr{F}_0)}^2\right)dt + 2\langle X_t^d, B_t^d(X_t^d)dW_t\rangle_{\mathscr{F}_0}
\end{aligned}
\end{equation}
and then 
\begin{equation*}\label{loginviscidgen}
\begin{aligned}
&d \log \|X_t^d\|_{\mathscr{F}_0}^2 =\frac{1}{\|X_t^d\|_{\mathscr{F}_0}^2} \left(2{}_\mathscr{D}\langle X_t^d, A^d(X_t^d)\rangle_{\mathscr{G}} + \|B_t^d(X_t^d)\|_{L_2(\mathscr{U}, \mathscr{F}_0)}^2\right)dt + \frac{2}{\|X_t^d\|_{\mathscr{F}_0}^2}\langle X_t^d, B_t^d(X_t^d)dW_t\rangle_{\mathscr{F}_0} \\
& -\frac{2}{\|X_t^d\|_{\mathscr{F}_0}^4}\langle X_t^d, B_t^d(X_t^d)\rangle_{\mathscr{F}_0}^2dt - \frac{2\epsilon}{\|X_t^d\|_{\mathscr{F}_0}^4}\langle X_t^d, B_t^d(X_t^d)\rangle_{\mathscr{F}_0}^2dt + \frac{2\epsilon}{\|X_t^d\|_{\mathscr{F}_0}^4}\langle X_t^d, B_t^d(X_t^d)\rangle_{\mathscr{F}_0}^2dt \\
& = \frac{1}{\|X_t^d\|_{\mathscr{F}_0}^2} \left(2{}_\mathscr{D}\langle X_t^d, A^d(X_t^d)\rangle_{\mathscr{G}} + \|B_t^d(X_t^d)\|_{L_2(\mathscr{U}, \mathscr{F}_0)}^2\right)dt - \frac{2(1-\epsilon)}{\|X_t^d\|_{\mathscr{F}_0}^4}\langle X_t^d, B_t^d(X_t^d)\rangle_{\mathscr{F}_0}^2dt \\
& + dN_t - \frac{\epsilon}{2}d\langle N \rangle_t
\end{aligned}
\end{equation*}
where
\begin{equation*}
N_t := \displaystyle\int_0^t \frac{2}{\|X_s^d\|_{\mathscr{F}_0}^2}\langle X_s^d, B_s^d(X_s^d) \rangle_{\mathscr{F}_0} dW_s
\end{equation*}
The drift above is given by 
\begin{equation*}
p_1:=\frac{2}{\|X_t^d\|_{\mathscr{F}_0}^4} \left( \|X_t^d\|_{\mathscr{F}_0}^2 {}_\mathscr{D}\langle X_t^d, A^d(X_t^d)\rangle_{\mathscr{G}} + \frac{1}{2}\|X_t^d\|_{\mathscr{F}_0}^2\|B_t^d(X_t^d)\|_{L_2(\mathscr{U}, \mathscr{F}_0)}^2 - (1-\epsilon)\langle X_t^d, B_t^d(X_t^d)\rangle_{\mathscr{F}_0}^2 \right)dt 
\end{equation*}
and by condition \eqref{cond:compressviscous} we have
\begin{equation*}
\begin{aligned}
&\frac{2}{\|X_t^d\|_{\mathscr{F}_0}^4} \left( \|X_t^d\|_{\mathscr{F}_0}^2 \langle X_t^d, A^d(X_t^d)\rangle_{\mathscr{F}_0} + \frac{1}{2}\|X_t^d\|_{\mathscr{F}_0}^2\|B_t^d(X_t^d)\|_{L_2(\mathscr{U}, \mathscr{F}_0)}^2 - (1-\epsilon)\langle X_t^d, B_t^d(X_t^d)\rangle_{\mathscr{F}_0}^2 \right)dt \\
& \leq \frac{2}{\|X_t^d\|_{\mathscr{F}_0}^4}\big(
C_1\|X_t^d\|_{\mathscr{F}_0}^{\gamma_1+2}-C_2\|X_t^d\|_{\mathscr{F}_0}^2\|X_t^d\|_{\mathscr{F}_1}^2 +\frac{1}{2}\theta^2\|X_t^d\|_{\mathscr{F}_1}^{2\alpha_1}\|X_t^d\|_{\mathscr{F}_0}^4 
- (1-\epsilon) \theta^2\|X_t^d\|_{\mathscr{F}_1}^{2\alpha_1}\|X_t^d\|_{\mathscr{F}_0}^4\big)dt \\
& \leq 2\left(\left(C_1+(\epsilon - \frac{1}{2})\theta^2)\right)\|X_t^d\|_{\mathscr{F}_1}^{2\alpha_1}\right)dt
\end{aligned}
\end{equation*}
and we can choose $\theta= \frac{2C_1}{1/2- \epsilon}$. This implies that 
\begin{equation*}\label{loginviscidgen'}
d \log \|X_t^d\|_{\mathscr{F}_0}
\le dN_t - \frac{\epsilon}{2}d\langle N \rangle_t
\end{equation*}
Since we chose
$\gamma_1-2<2\alpha_1$, we have 
\begin{equation}
2C_1\|X_t^d\|_{\mathscr{F}_0}^{\gamma_1-2} + \left(C_1 + (\epsilon - \frac{1}{2})\theta^2\right)\|X_t^d\|_{\mathscr{F}_1}^{2\alpha_1}\le 
2C_1\|X_t^d\|_{\mathscr{F}_1}^{\gamma_1-2} + \left(C_1 + (\epsilon -\frac{1}{2}\right)\theta^2)\|X_t^d\|_{\mathscr{F}_1}^{2\alpha_1}
\end{equation}
due to the fact that $\mathscr{F}_1 \hookrightarrow \mathscr{F}_0$ (and we are in a $d$-dimensional space),
which is a polynomial in 
$\|X_t^d\|_{\mathscr{F}_0}$ with negative highest coefficient which is bounded (as before). 
So we have
\begin{equation}
\begin{aligned}
p_1 \leq 2C_1\|X_t^d\|_{\mathscr{F}_0}^{\gamma_1-2} + \left(C_1 + (\epsilon - \frac{1}{2})\theta^2\right)\|X_t^d\|_{\mathscr{F}_0}^{2\alpha_1} dt
\end{aligned}
\end{equation}
and after replacing this in \eqref{loginviscidgen} and using the same argument as per Revuz \& Yor as before we obtain 
\begin{equation}
\mathbb{P}\left( \displaystyle\sup_{t\in[0,T]}\|X_t^d\|_{\mathscr{F}_0}^2 \geq \beta_{11}\right) < \delta_{11}.
\end{equation}
Now in order to obtain
\begin{equation}
\mathbb{P}\left( \displaystyle\sup_{t\in[0,T]}\displaystyle\int_0^t\|X_t^d\|_{\mathscr{F}_1}^{2\alpha_1}ds \geq \beta_{22}\right) < \delta_{22}
\end{equation}
the arguments are similar to those from the previous section (viscous case). 
Therefore $\|X_t^d\|_{\mathscr{F}_0}$ is controlled overall as before (by an exponential, after considering the stochastic integral etc) and we go back to \eqref{aneqngen} to now control $\displaystyle\int_0^t \|X_s^d\|_{\mathscr{F}_1}^2 ds$.
Note that these calculations are in some sense similar in spirit with those from the previous case, although $C(\theta, \alpha_2)$ in 
\begin{equation}\label{eq:last}
\begin{aligned}
d\|X_t^d\|_{\mathscr{F}_0}^2 - C(\theta, \alpha_1)\|X_t^d\|_{\mathscr{F}_1}^2dt \leq G(\|X_t^d\|_{\mathscr{F}_0})dt + \tilde{B}(X_t^d)dW_t
\end{aligned}
\end{equation}
is now positive, but the role played by the $\eta$-viscous term in the previous case is now played by the term $\frac{\epsilon}{2}\langle N\rangle_t$ which was added 'artificially'.

\noindent\textbf{Case III:} the diffusion coefficient is given by
\begin{equation}
B_t^d(X_t^d) = \theta \|X_t^d\|_{\mathscr{F}_0}^{\alpha_0}X_t^d.
\end{equation}
Then
\begin{equation}
\begin{aligned}
\|B_t^d(X_t^d)\|_{L_2(\mathscr{U}, \mathscr{F}_1)}^2 
& = \theta^2\|X_t^d\|_{\mathscr{F}_0}^{2\alpha_0}\|X_t^d\|_{\mathscr{F}_1}^2 \leq \theta^2\|X_t^d\|_{\mathscr{F}_1}^{2\alpha_0+2}
\end{aligned}
\end{equation}
\begin{equation}
\begin{aligned}
\tilde{B}_t^d(X_t^d)^2 &:=(2\langle X_t^d, B^d(X_t^d)\rangle_{\mathscr{F}_1})^2 = (2\theta\|X_t^d\|_{\mathscr{F}_0}^{\alpha_0}\|X_t^d\|_{\mathscr{F}_1}^2)^2 \\
& = 4\theta^2\|X_t^d\|_{\mathscr{F}_0}^{2\alpha_0}\|X_t^d\|_{\mathscr{F}_1}^4 \leq 4\theta^2\|X_t^d\|_{\mathscr{F}_1}^{2\alpha_0+4}
\end{aligned}
\end{equation}
and

\begin{equation}
\begin{aligned}
\tilde{A}^d(X^d) &=2{}_\mathscr{D}\langle X_t^d, A^d(X_t^d)\rangle_{\mathscr{G}} + \|B_t(X_t^d)\|_{L_2(\mathscr{U}, \mathscr{F}_1)}^2 \\
%& = 2{}_\mathscr{D}\langle X_t^d, A^d(X_t^d)\rangle_{\mathscr{G}} + \theta^2\|X_t^d\|_{\mathscr{F}_0}^{2\alpha_0+2} \\
& = 2\langle X_t^d, A^d(X_t^d)\rangle_{\mathscr{F}_1} + \theta^2\|X_t^d\|_{\mathscr{F}_0}^{2\alpha_0}\|X_t^d\|_{\mathscr{F}_1}^2\\
& \leq 2C_1 \|X_t^d\|_{\mathscr{F}_0}^{\gamma_{13}}\|X_t^d\|_{\mathscr{F}_1}^2+ \theta^2\|X_t^d\|_{\mathscr{F}_0}^{2\alpha_0}\|X_t^d\|_{\mathscr{F}_1}^2 \\
& \leq 2C_1 \|X_t^d\|_{\mathscr{F}_0}^{\gamma_{13}}\|X_t^d\|_{\mathscr{F}_1}^2+ \theta^2\|X_t^d\|_{\mathscr{F}_1}^{2\alpha_0+2}.
\end{aligned}
\end{equation}
due to conditions \eqref{condscproduct} and \eqref{cond:incompress}.
Then  \eqref{logeqn} becomes

\begin{equation}
\begin{aligned}
d \log  & \left( C + \|X_t^d\|_{\mathscr{F}_1}^2 \right) = \frac{1}{C+ \|X_t^d\|_{\mathscr{F}_1}^2}\left(\tilde{A}_t^d(X_t^d)dt + \tilde{B}_t^d(X_t^d)dW_t^d \right) - \frac{1}{2(C+ \|X_t^d\|_{\mathscr{F}_1}^2)^2}\tilde{B}_t^d(X_t^d)^2dt \\
& = \frac{1}{C+\|X_t^d\|_{\mathscr{F}_1}^2}\tilde{A}_t^d(X_t^d)dt - \frac{1-\epsilon}{2(C+ \|X_t^d\|_{\mathscr{F}_1}^2)^2}\tilde{B}_t
^d(X_t^d)^2dt \\
& + \frac{1}{C+\|X_t^d\|_{\mathscr{F}_1}^2}\tilde{B}_t^d(X_t^d)dW_t^d  - \frac{\epsilon}{2(C+\|X_t^d\|_{\mathscr{F}_1}^2)^2}\tilde{B}_t^d(X_t^d)^2dt\\
& =\frac{1}{\left(C+\|X_t^d\|_{\mathscr{F}_1}^2\right)^2}\left( (C+\|X_t^d\|_{\mathscr{F}_1}^2)\tilde{A}_t^d(X_t^d)-\frac{1-\epsilon}{2}\tilde{B}_t
^d(X_t^d)^2 \right)dt + dM_t - \frac{\epsilon}{2}d\langle M \rangle_t \\
& \leq \frac{1}{\left(C+\|X_t^d\|_{\mathscr{F}_1}^2\right)^2} \left( (C+\|X_t^d\|_{\mathscr{F}_1}^2) (2C_1 \|X_t^d\|_{\mathscr{F}_0}^{\gamma_{13}}\|X_t^d\|_{\mathscr{F}_1}^2 \right. \\
& \left. + \theta^2\|X_t^d\|_{\mathscr{F}_0}^{2\alpha_0}\|X_t^d\|_{\mathscr{F}_1}^2) - \frac{1-\epsilon}{2}4\theta^2\|X_t^d\|_{\mathscr{F}_0}^{2\alpha_0}\|X_t^d\|_{\mathscr{F}_1}^4\right)dt + dM_t - \frac{\epsilon}{2}d\langle M \rangle_t 
\end{aligned}
\end{equation}
\begin{equation}
\begin{aligned}
& = \frac{1}{\left(C+\|X_t^d\|_{\mathscr{F}_1}^2\right)^2} \left(\tilde{C}_1\|X_t^d\|_{\mathscr{F}_0}^{\gamma_{13}}\|X_t^d\|_{\mathscr{F}_1}^2 + 2C_1\|X_t^d\|_{\mathscr{F}_0}^{\gamma_{13}}\|X_t^d\|_{\mathscr{F}_1}^4+ \theta^2\|X_t^d\|_{\mathscr{F}_0}^{2\alpha_0}\|X_t^d\|_{\mathscr{F}_1}^2 \right. \\
& \left. - 2(1-\epsilon)\theta^2\|X_t^d\|_{\mathscr{F}_0}^{2\alpha_0}\|X_t^d\|_{\mathscr{F}_1}^4\right)dt  + dM_t - \frac{\epsilon}{2}d\langle M \rangle_t \\
& \leq \frac{\|X_t^d\|_{\mathscr{F}_1}^2\|X_t^d\|_{\mathscr{F}_0}^{2\alpha_0}}{\left(C+\|X_t^d\|_{\mathscr{F}_1}^2\right)^2} \left( (\tilde{C}_1 + \theta^2) + (2C_1 - 2(1-\epsilon)\theta^2)\|X_t^d\|_{\mathscr{F}_1}^2\right) + dM_t - \frac{\epsilon}{2}d\langle M \rangle_t 
\end{aligned}
\end{equation}
and we choose $\theta$ and $\epsilon \in (0,1)$ such that $2C_1 - 2(1-\epsilon)\theta^2 < 0$. Also, $\alpha_0 =\gamma_{13}/2$.

\end{proof}

\begin{remark}
The previous proposition provides sufficient control conditions in order to pass to the limit in the stochastic term in the case of a compressible viscous model and in the case of an incompressible model, as presented in this paper. However, in order for the approximating stochastic integral to converge in the compressible inviscid model case too, we need uniform control also in $\mathcal{D}$. This is dictated by the special form of the diffusion term which now contains the $\mathscr{F}_1$ norm of the solution. \footnote{Note that the control for $\displaystyle\sup_{t\in[0,T]}\|X_t^d\|_{\mathscr{F}_0}$ and $\displaystyle\int_0^T\|X_t^d\|_{\mathscr{F}_1}$ still holds as before also in the compressible inviscid case, but these two conditions are not sufficient.} We prove that such a control holds $\mathbb{P}$ - a.s. in the following proposition.
\end{remark}

\begin{proposition}\label{prop:unifcontrolinD}
The approximating sequence $(X_t^d)_d$ is $\mathbb{P}$ - a.s. uniformly controlled in $\mathscr{D}$, that is, for any $\epsilon_2 >0$ there exists $\mathcal{K}_2 >0$ such that
\begin{equation}\label{eq:unifcontrolinD}
\displaystyle\sup_d \mathbb{P}\left( \displaystyle\sup_t \|X_t^d\|_{\mathscr{D}}^2 \geq \mathcal{K}_2\right) \leq \epsilon_2.
\end{equation}
\end{proposition}

\begin{proof}

The approximating equation 
\begin{equation}
dX_t^d = A_t^d(X_t^d)dt + B_t^d(X_t^d)dW_t^d \quad \hbox{in} 
\end{equation}
holds in $\mathscr{D}^d$ endowed with the scalar product inherited from $\mathscr{D}$.
By the It\^{o} formula 
\begin{equation}
\begin{aligned}
d\varphi(\|X_t^d\|_{\mathscr{D}}) &= \frac{\partial\varphi}{\partial x}(\|X_t^d\|_{\mathscr{D}})\tilde{A}^d(X_t^d)dt + \frac{\partial\varphi}{\partial x}\tilde{B}^d(\|X_t^d\|_{\mathscr{D}})dW_t^d + \frac{1}{2}\frac{\partial^2\varphi}{\partial x^2}(\|X_t^d\|_{\mathscr{D}})(\tilde{B}^d(X_t^d))^2dt
\end{aligned}
\end{equation}
in $(\mathscr{D}^d, \langle \cdot, \cdot \rangle_{\mathscr{D}})$, 
and we take
\begin{equation}
\varphi(\|x\|_{\mathscr{D}}):= \log(\|x\|_{\mathscr{D}}^2). 
\end{equation}

The uniform control holds by the same arguments as those used in the proof of Proposition \ref{prop:unifcontrol}, with the only difference that here we use condition \eqref{cond:compressinviscid}. There is, however, a subtlety in condition \eqref{cond:compressinviscid}: as opposed to conditions \eqref{cond:compressviscous} and \eqref{cond:incompress}, in \eqref{cond:compressinviscid} we do not require existence of a functional space which is of higher order compared to $\mathscr{D}$, on the right hand side. This is possible only in $\mathscr{D}$ and the justification hinges upon the fact that $\mathscr{D}^d$ is a finite-dimensional space (so all norms are equivalent), and the Fatou lemma.

\end{proof}

\subsection{Tightness}\label{sect:tightness}
\begin{proposition}\label{tightness}
The approximating sequence of solutions $(X_t^d)_d$ is tight in $C\left([0, T],\mathscr{G} \right)$.
\end{proposition}

\begin{proof}
In order to prove the tightness result, we use Theorem \ref{aldouslemma} and Proposition \ref{prop:unifcontrol}. We first show that for every $t\in \lbrack 0,T]$, the laws of $(X_t^d)_d$ are tight in $\mathscr{G}$. Then we prove that the sequence $(X^{d})_{d=1}^{\infty }$ satisfies the
Aldous condition \ref{cond1}. 

To prove tightness in $\mathscr{G}$ of the laws of $\left( X_t^{d}\right)_{d=1}^{\infty }$ for every $t\in \lbrack 0,T]$, it suffices to show that that for any $\epsilon >0,$ there exists a compact sets $K_{\epsilon }\in \mathscr{G}$
such that 
\begin{equation*}
\sup_{d\geq 1}\mathbb{P}\left(X_t^{d} \in \mathscr{G}\backslash K_{\epsilon }\right) \leq \epsilon.
\end{equation*}
Following from Proposition \ref{prop:unifcontrol}, we deduce that there
exists $\mathcal{K}_{\epsilon }$ such that 
\begin{equation*}
\sup_{d}\mathbb{P}\left( \Vert X_{t}^{d}\Vert _{\mathscr{F}_{0}}\geq 
\mathcal{K}_{\epsilon }\right) \leq \epsilon_{1}.
\end{equation*}%
Let $K_{\epsilon }=\left\{ x\in \mathscr{G}|\Vert x\Vert _{\mathscr{F}%
_{0}}\leq \mathcal{K}_{\epsilon }\right\} \subset \mathscr{G}$. Since $%
\mathscr{F}_{0}$ is compactly embedded in $\mathscr{G}$ it follows that the
set $K_{\epsilon }$ is compact in $\mathscr{G}$ and 
\begin{equation*}
\sup_{d\geq 1}\mathbb{P}\left(X_t^{d}\in \mathscr{G}\backslash K_{\epsilon }\right) \leq \sup_{d}\mathbb{P}\left( \Vert
X_{t}^{d}\Vert _{\mathscr{F}_{0}}\geq \mathcal{K}_{\epsilon }\right) \leq
\epsilon _{1}
\end{equation*}%
as required. 
We show below that the equivalent of condition \eqref{cond1} holds, with $\rho=\|\cdot \|_{\mathcal{G}}$.
Let $\tau^d$ be a positive stopping time. We need to check that for any $\epsilon >0,$ there exists $\delta >0$ such that for every sequence $(\tau_{d})_{d=1}^{\infty }$ of $\left( \mathcal{F}_{t}\right)
_{t\geq 0^{-}}$ - stopping times with $\tau_{d}\leq T$ 
\begin{equation*}
\sup_{d\geq 1}\sup_{0\leq t\leq \delta }\mathbb{P}\left(\left\vert
\left\vert X_{\tau_{d}+t}^{d} -X_{\tau_{d}}^{d}  \right\vert \right\vert_{\mathscr{G}}\geq \eta \right) \leq \epsilon .
\end{equation*}
We have
\begin{eqnarray*}
\Vert X_{\tau _{d}+t}^{d}-X_{\tau _{d}}^{d}\Vert _{\mathscr{G}}^{2} &\leq
&2\left\Vert \int_{\tau _{d}}^{\tau _{d}+t}A_{s}^{d}(X_{s}^{d})ds\right\Vert
_{\mathscr{G}}^{2}+2 \left\Vert \int_{\tau _{d}}^{\tau
_{d}+t}B_{s}^{d}(X_{s}^{d})dW_{s}\right\Vert _{\mathscr{G}}^{2} \\
&\leq & 2t \int_{\tau _{d}}^{\tau _{d}+t}\left\Vert
A_{s}^{d}(X_{s}^{d})\right\Vert _{\mathscr{G}}^{2}ds + 2\left\Vert \int_{\tau
_{d}}^{\tau _{d}+t}B_{s}^{d}(X_{s}^{d})dW_{s}\right\Vert _{\mathscr{G}}^{2}
\\
&\leq &2\delta \int_{\tau _{d}}^{\tau _{d}+t}\left\Vert
A_{s}^{d}(X_{s}^{d})\right\Vert _{\mathscr{G}}^{2}ds + 2 \left\Vert \int_{\tau
_{d}}^{\tau _{d}+t}B_{s}^{d}(X_{s}^{d})dW_{s}\right\Vert _{\mathscr{G}}^{2}
\end{eqnarray*}
We prove first that for any $\epsilon ,\eta >0,$ there exists $\delta >0$
such that 
\begin{equation*}
\sup_{d\geq 1}\sup_{0\leq t\leq \delta }\mathbb{P}\left( \delta \int_{\tau
_{d}}^{\tau _{d}+t}\left\Vert A_{s}^{d}(X_{s}^{d})\right\Vert _{\mathscr{G}}^{2}ds\geq \eta \right) \leq \epsilon .
\end{equation*}
We divide the argument into three cases:

\noindent\textbf{Case I: }We have that\textbf{\ } 
\begin{eqnarray*}
\int_{\tau _{d}}^{\tau _{d}+t}\left\Vert A_{s}^{d}(X_{s}^{d})\right\Vert _{%
\mathscr{G}}^{2}ds &\leq &C\left( \int_{\tau _{d}}^{\tau _{d}+t}\left\Vert
X_{s}^{d}\right\Vert _{\mathscr{F}_{0}}^{2\gamma }ds+C\int_{\tau _{d}}^{\tau
_{d}+t}\left\Vert X_{s}^{d}\right\Vert _{\mathscr{F}_{1}}^{2}ds+1\right)  \\
&\leq &C\left( T\sup_{t\in \left[ 0,T\right] }\left\Vert
X_{s}^{d}\right\Vert _{\mathscr{F}_{0}}^{2\gamma }+\int_{0}^{T}\left\Vert
X_{s}^{d}\right\Vert _{\mathscr{F}_{1}}^{2}ds+1\right) 
\end{eqnarray*}%
From the boundedness result, we choose $R>0$ such that 
\begin{eqnarray*}
\mathbb{P}\left(\sup_{s\in \lbrack 0,T]}\Vert X_{t}^{d}\Vert _{%
\mathscr{F}_{0}}\geq R \right)  &\leq &\frac{\epsilon }{2}. \\
\mathbb{P}\left( \int_{0}^{T}\left\Vert X_{s}^{d}\right\Vert _{%
\mathscr{F}_{1}}^{2}ds\geq R \right)  &\leq &\frac{\epsilon }{2}
\end{eqnarray*}
It follows that 
\begin{equation*}
\mathbb{P}\left(\left\Vert \int_{\tau _{d}}^{\tau
_{d}+t}A_{s}^{d}(X_{s}^{d})ds\right\Vert _{\mathscr{G}}^{2}\geq C\left(
TR^{2\gamma }+R+1\right) \right) \leq \epsilon 
\end{equation*}%
We choose  $\delta $ such that 
\begin{equation*}
\delta \left( C\left( TR^{2\gamma }+R+1\right) \right) =\eta \Leftrightarrow
\delta =\frac{\eta }{C\left( TR^{2\gamma }+R+1\right) }
\end{equation*}
which gives us
\begin{equation*}
\mathbb{P}\left(\delta \int_{\tau _{d}}^{\tau _{d}+t}\left\Vert
A_{s}^{d}(X_{s}^{d})\right\Vert _{\mathscr{G}}^{2}ds\geq \eta \right) =%
\mathbb{P}\left( \left\Vert \int_{\tau _{d}}^{\tau
_{d}+t}A_{s}^{d}(X_{s}^{d})ds\right\Vert _{\mathscr{G}}^{2}\geq C\left(
TR^{2\gamma }+R+1\right) \right) \leq \epsilon.
\end{equation*}

\noindent \textbf{Case II: }We have that\textbf{\ } 
\begin{equation*}
\int_{\tau _{d}}^{\tau _{d}+t}\left\Vert A_{s}^{d}(X_{s}^{d})\right\Vert _{%
\mathscr{G}}^{2}ds\leq C\left( T\sup_{t\in \left[ 0,T\right] }\left\Vert
X_{s}^{d}\right\Vert _{\mathscr{D}}^{2\gamma }+1\right) 
\end{equation*}%
and the argument follows similar to case \textbf{Case I}.

\noindent\textbf{Case III: }We have that\textbf{\ } 
\begin{equation*}
\int_{\tau _{d}}^{\tau _{d}+t}\left\Vert A_{s}^{d}(X_{s}^{d})\right\Vert _{%
\mathscr{G}}^{2}ds\leq C\left( T\sup_{t\in \left[ 0,T\right] }\left\Vert
X_{s}^{d}\right\Vert _{\mathscr{F}_{1}}^{2\gamma }+1\right) 
\end{equation*}%
and the argument follows similar to case \textbf{Case I}.

\noindent For the stochastic integral we use Proposition \ref{gawareckimandrekar}. 
Let 
\begin{equation}
M_{\tau_d, \tau_{d+t}} := \displaystyle\int_{\tau_d}^{\tau_d+t} B_s^d(X_s^d)dW_s \quad \langle M \rangle_{\tau_d,\tau_d+t} = \displaystyle\int_{\tau_d}^{\tau_d+t} (B_s^d(X_s^d))^2ds.
\end{equation}
We prove that for any $\epsilon >0,$ there exists $\delta >0$ such that for every sequence $\left(\tau_{d}\right)_{d=1}^{\infty}$ of $\left( \mathcal{F}%
_{t}\right)_{t\geq 0^{-}}$ - stopping times with $\tau_d\leq T$ 
\begin{equation}
\sup_{d\geq 1}\sup_{0\leq t\leq \delta }\mathbb{P}\left(\left\vert
\left\vert M_{\tau _{d},\tau _{d+t}}\right\vert \right\vert _{\mathscr{G}%
}\geq \eta \right) \leq \epsilon .  \label{claim2}
\end{equation}
We use here Proposition \ref{gawareckimandrekar} to deduce that  for any $%
\eta ,n>0$, 
\begin{equation*}
\mathbb{P}\left( \left\vert \left\vert M_{\tau _{d},\tau _{d+t}}\right\vert
\right\vert _{\mathscr{G}}\geq \eta \right) \leq \frac{n}{\eta ^{2}}+\mathbb{%
P}\left( \langle M\rangle _{\tau _{d},\tau _{d}+t}>n\right) \leq \frac{n}{%
\eta ^{2}}+\mathbb{P}\left( \delta \sup_{t\in \left[ 0,T\right] }\left\vert
\left\vert B_{t}^{d}(X_{t}^{d})\right\vert \right\vert _{\mathscr{G}%
}^{2}>n\right) 
\end{equation*}%
and if we choose $n=\eta ^{2}\sqrt{\delta }$ we get that 
\begin{equation*}
\mathbb{P}\left( \left\vert \left\vert M_{\tau _{d},\tau _{d+t}}\right\vert
\right\vert _{\mathscr{G}}\geq \eta \right) \leq \sqrt{\delta }+\mathbb{P}%
\left( \sup_{t\in \left[ 0,T\right] }\left\vert \left\vert
B_{t}^{d}(X_{t}^{d})\right\vert \right\vert _{\mathscr{G}}^{2}>\frac{\eta
^{2}}{\sqrt{\delta }}\right). 
\end{equation*}

\noindent\textbf{Case I and III:} Here we assume that $B_{t}(X_{t})=\theta \Vert
X_{t}\Vert _{\mathscr{F}_{0}}^{\alpha _{0}}X_{t}$. It follows that 
\begin{equation*}
\sup_{t\in \left[ 0,T\right] }\left\vert \left\vert
B_{s}^{d}(X_{s}^{d})\right\vert \right\vert _{\mathscr{G}}^{2}=\theta
^{2}\sup_{t\in \left[ 0,T\right] }\Vert X_{t}^d\Vert _{\mathscr{F}%
_{0}}^{2\alpha _{0}}\left\vert \left\vert X_{t}^{d}\right\vert
\right\vert _{\mathscr{G}}^{2}\leq \theta ^{2}\sup_{t\in \left[ 0,T\right]
}\Vert X_{t}\Vert _{\mathscr{F}_{0}}^{2\alpha _{0}+2}
\end{equation*}%
hence%
\begin{equation*}
\mathbb{P}\left( \left\vert \left\vert M_{\tau _{d},\tau _{d+t}}\right\vert
\right\vert _{\mathscr{G}}\geq \eta \right) \leq \sqrt{\delta }+\mathbb{P}%
\left( \sup_{t\in \left[ 0,T\right] }\Vert X_{t}\Vert _{\mathscr{F}%
_{0}}>\left( \frac{\eta ^{2}}{\sqrt{\delta }\frac{\eta ^{2}}{\sqrt{\delta }}}%
\right) ^{\frac{1}{2\alpha _{0}+2}}\right) 
\end{equation*}%
and the result follows as above, by choosing $\delta \leq \left( \frac{%
\epsilon }{2}\right) ^{2}$ \emph{and} \ such that 
\begin{equation*}
\mathbb{P}\left( \sup_{t\in \left[ 0,T\right] }\Vert X_{t}\Vert _{\mathscr{F}%
_{0}}^{2\alpha _{0}+2}>\left( \frac{\eta ^{2}}{\sqrt{\delta }\frac{\eta ^{2}%
}{\sqrt{\delta }}}\right) ^{\frac{1}{2\alpha _{0}+2}}\right) \leq \frac{%
\epsilon }{2}
\end{equation*}%
again by using Proposition \ref{prop:unifcontrol}. 

\noindent\textbf{Case II:} Here we assume that $B_{t}(X_{t})=\theta \Vert X_{t}\Vert
_{\mathscr{F}_{1}}^{\alpha _{1}}X_{t}$. It follows that 
\begin{equation*}
\sup_{t\in \left[ 0,T\right] }\left\vert \left\vert
B_{s}^{d}(X_{s}^{d})\right\vert \right\vert _{\mathscr{G}}^{2}=\theta
^{2}\sup_{t\in \left[ 0,T\right] }\Vert X_{t}\Vert _{\mathscr{F}%
_{1}}^{2\alpha _{1}}\left\vert \left\vert B_{t}^{d}(X_{t}^{d})\right\vert
\right\vert _{\mathscr{G}}^{2}\leq \theta ^{2}\sup_{t\in \left[ 0,T\right]
}\Vert X_{t}\Vert _{\mathscr{F}_{1}}^{2\alpha _{1}+2} 
\end{equation*}%
The analysis is then identical with that of \textbf{Case I and III }by applying 
Proposition \ref{prop:unifcontrolinD} instead of Proposition \ref%
{prop:unifcontrol}. The tightness of the sequence is now complete.
\end{proof}

\vspace{5mm}

\subsection{Applications to models from stochastic fluid dynamics}\label{sect:applications}

\subsubsection{Viscous compressible and incompressible models}
\textbf{Incompressible Navier Stokes equation in vorticity form} on the 3D torus ($\mathbb{T}^{3}$)  
\begin{equation}
d_{t}\omega +\mathcal{L}_{v}\omega =\Delta \omega   \label{eq_NS}
\end{equation}%
Here $\mathcal{L}_{v}\omega $ denotes the Lie derivative with respect to
the velocity field $v$. That is 
\begin{equation}
\mathcal{L}_{v}\omega :=(v\cdot \nabla )\omega -(\omega \cdot \nabla )v=:%
\left[ \,v\,,\,\omega \,\right] .  \label{eq Lie brkt}
\end{equation}%
The velocity field $v$ \ is obtain via the Biot-Savart operator which
reconstructs a zero mean divergence free vector field $u$ from a divergence
free vector field $\omega $ such that $\mathrm{curl}u=\omega $. On the torus
it is given by $u=-\mathrm{curl}\Delta ^{-1}\omega $.

Here $\mathscr{G}:=L^{2}\left( \mathbb{T}^{3};\mathbb{R}^{3}\right) $ is the space of square integrable zero mean divergence free
three dimensional vector fields, $\mathscr{F}_{0}:=W^{2,2}\left(\mathbb{T}^{3};\mathbb{R}^{3}\right) ,\mathscr{F}_{1}:=W^{3,2}\left( \mathbb{T}^{3};\mathbb{R}^{3}\right) ,$ $\mathscr{D}%
:=W^{4,2}\left( \mathbb{T}^{3};\mathbb{R}^{3}\right).$ In this
case, the operator $A$ is given by 
\begin{equation*}
A\left( \omega \right) :=-\left[ \,v \,,\,\omega \,%
\right] +\Delta \omega 
\end{equation*}%

In this case, one can check that for $\omega \in \mathscr{D}:=W^{4,2}\left( \mathbb{T}^{3};\mathbb{R}^{3}\right)$ there exists $\alpha
,\beta >0$ such that\  
\begin{equation*}
\left\langle \omega ,A\left( \omega \right) \right\rangle _{W^{2,2}\left( \mathbb{T}^{3};\mathbb{R}^{3}\right) }\leq \alpha \left\vert
\left\vert \omega \right\vert \right\vert _{W^{2,2}\left( \mathbb{T}^{3};\mathbb{R}^{3}\right) }^{6}-\beta \left\vert \left\vert \omega
\right\vert \right\vert _{W^{3,2}\left( \mathbb{T}^{3};\mathbb{R}^{3}\right)}^{2}
\end{equation*}%
Moreover%
\begin{equation*}
\left\vert \left\vert A\left( \omega \right) \right\vert \right\vert
_{L^{2}\left( \mathbb{T}^{3};\mathbb{R}^{3}\right) }\leq c\left(
\left\vert \left\vert \omega \right\vert \right\vert _{W^{2,2}\left( \mathbb{T}^{3};\mathbb{R}^{3}\right) }^{\frac{3}{2}%
}+\left\vert \left\vert \omega \right\vert \right\vert _{W^{2,2}\left( \mathbb{T}^{3};\mathbb{R}^{3}\right) }\right) 
\end{equation*}%
which implies that $A$ is bounded on the compact sets $K_{\epsilon }=\left\{
\Vert \omega \Vert _{W^{2,2}\left( \mathbb{T}^{3};\mathbb{R}%
^{3}\right) }\leq R\right\} \in L_{\sigma }^{2}\left( \mathbb{T}^{3};\mathbb{%
R}^{3}\right) $(this is required for tightness).\ Also the interpolation
property holds for $L^{2}\left( \mathbb{T}^{3};\mathbb{R}%
^{3}\right) ,$ $W^{2,2}\left( \mathbb{T}^{3};\mathbb{R}^{3}\right) 
$ and $W^{3,2}\left( \mathbb{T}^{3};\mathbb{R}^{3}\right) $.

\vspace{3mm}
\noindent\textbf{Viscous Rotating Shallow Water Model (2D)}\\
Let $u=(u^{1},u^{2})$ be the horizontal fluid velocity of the shallow water model, defined on the two-dimensional torus $\mathbb{T}^{2}$ and $h$ be the
thickness of the fluid column. Let also $v:=\epsilon u+\mathcal{R}$, with $%
curl\ \mathcal{R}=f\hat{z}$. $\mathcal{R}$ corresponds to the vector
potential for the (divergence-free) rotation rate about the vertical
direction, and it is chosen here such that $\nabla \mathcal{R}=0$. We will
denote by $a:=\left( v,h\right) $ the solution of the viscous rotating shallow water (RSW) system 
\begin{equation*}
{\mathrm{d}}_{\mathrm{t}}a_{t}=F\left( a_{t}\right) ,
\end{equation*}%
where $F\left( a_{t}\right) $ denotes 
\begin{equation*}
F\left( 
\begin{array}{c}
v \\ 
h%
\end{array}%
\right) =\left( 
\begin{array}{c}
-u\cdot \nabla v-f\hat{z}\times u-\nabla p+\gamma \Delta v_{t} \\ 
-\nabla \cdot (hu)+\eta \Delta h_{t}%
\end{array}%
\right)
\end{equation*}%
and
\begin{itemize}
\item $\epsilon $ is the Rossby number, a dimensionless number which
describes the effects of rotation on the fluid flow: a small Rossby number ($%
\epsilon <<1$) suggests that the rotation dominates over the advective
terms; it can be expressed as $\epsilon =\frac{U}{fL}$ where $U$ is a
typical scale for horizontal speed and $L$ is a typical length scale.

\item $f$ is the Coriolis parameter, $f=2\Theta \sin \varphi $ where $\Theta 
$ is the rotation rate of the Earth and $\varphi $ is the latitude; $f\hat{z}%
\times u=(-fu^{2},fu^{1})$, where $\hat{z}$ is a unit vector pointing away
form the centre of the Earth. For the analytical analysis we assume $f$ to
be constant.

\item $p:=\frac{h-b}{\epsilon \mathscr{F}_0}$, $\nabla p$ is the pressure
term, $b$ is the bottom topography function.

\item $\mathscr{F}$ is the Froude number, a dimensionless number which
relates to the stratification of the flow. It can be expressed as $%
\mathscr{F}=\frac{U}{NH}$ where $H$ is the typical vertical scale and $N$ is
the buoyancy frequency.
\item $\gamma, \eta$ are viscosity coefficients.
\end{itemize}

\bigskip Here $\mathscr{G}:=L^{2}\left( \mathbb{T}^{2};\mathbb{R}^{2}\right) 
$ is the space of square integrable two dimensional vector fields, $%
\mathscr{F}_{0}:=W^{1,2}\left( \mathbb{T}^{2};\mathbb{R}^{2}\right) ,%
\mathscr{F}_{1}:=W^{2,2}\left( \mathbb{T}^{2};\mathbb{R}^{2}\right) ,$ $%
\mathscr{D}:=W^{3,2}\left( \mathbb{T}^{2};\mathbb{R}^{2}\right).$ Then one
can show that %(see formula 25 in the paper or maybe we do it by hand).%
\begin{equation}
\left\langle a,F\left( a\right) \right\rangle _{W^{1,2}\left( \mathbb{T}^{2};%
\mathbb{R}^{2}\right) }\leq C\left\vert \left\vert a\right\vert \right\vert
_{W^{1,2}\left( \mathbb{T}^{2};\mathbb{R}^{2}\right) }^{6}-\epsilon
\left\vert \left\vert a\right\vert \right\vert _{W^{2,2}\left( \mathbb{T}%
^{2};\mathbb{R}^{2}\right)}^{2}.  \label{term1}
\end{equation}%
Moreover%
\begin{equation*}
\left\vert \left\vert F\left( a\right) \right\vert \right\vert _{L^{2}\left( 
\mathbb{T}^{2};\mathbb{R}^{2}\right) }\leq c\left( \left\vert \left\vert
a\right\vert \right\vert _{W^{1,2}\left( \mathbb{T}^{2};\mathbb{R}%
^{2}\right) }^{2}+\left\vert \left\vert a\right\vert \right\vert
_{W^{2,2}\left( \mathbb{T}^{2};\mathbb{R}^{2}\right) }+1\right).
\end{equation*}%

\subsubsection{Inviscid compressible models}
\noindent\textbf{Burgers' equation} on the 2D torus ($\mathbb{T}^2$)
\begin{subequations}
\begin{equation}
d_{t}u+u\cdot \nabla u=0  \label{2Dburgers}
\end{equation}%
In this case, the operator $A$ is given by 
\end{subequations}
\begin{equation*}
A\left( u\right) :=-u\cdot \nabla u
\end{equation*}
and $\mathscr{G}:=L^{2}\left( \mathbb{T}^{3};\mathbb{R}^{3}\right)$,  $\mathscr{F}_{0}:=W^{1,2}\left( \mathbb{T}^{2};\mathbb{R}^{2}\right),\mathscr{F}_{1}:=W^{3,2}\left( \mathbb{T}^{2};\mathbb{R}^{2}\right)$ $\mathscr{D}:=W^{4,2}\left( \mathbb{T}^{2};\mathbb{R}^{2}\right).$
We have
\begin{eqnarray*}
\left\langle u,A\left( u\right) \right\rangle _{W^{1,2}\left( 
\mathbb{T}^{2};\mathbb{R}^{2}\right) } &\leq & C \left\vert \left\vert
u\right\vert \right\vert _{W^{3,2}\left( \mathbb{T}^{2};\mathbb{R}%
^{2}\right) }\left\vert \left\vert u\right\vert \right\vert _{W^{1,2}\left( \mathbb{T}^{2};\mathbb{R}^{2}\right) }^{2} \\
\left\vert \left\vert A\left( u\right) \right\vert \right\vert _{L^{2}\left( \mathbb{T}^{2};\mathbb{R}^{2}\right) } &\leq &\left\vert
\left\vert u\right\vert \right\vert _{W^{1,2}\left( \mathbb{T}^{2};%
\mathbb{R}^{2}\right)}^{2} \\
\left\langle u,A\left( u\right) \right\rangle _{W^{4,2}\left( 
\mathbb{T}^{2};\mathbb{R}^{2}\right) } &\leq &\left\vert \left\vert
u\right\vert \right\vert _{W^{3,2}\left( \mathbb{T}^{2};\mathbb{R}%
^{2}\right) }\left\vert \left\vert u\right\vert \right\vert _{W^{4,2}\left(\mathbb{T}^{2};\mathbb{R}^{2}\right)}^{2}.
\end{eqnarray*}

Note that \eqref{cond:compressviscous} can be obtained from the first inequality by a direct application of Young's inequality.

\noindent\textbf{Inviscid Rotating Shallow Water Model (2D)}

\noindent The rotating shallow water model (RSW) is as above, except the fact that there is no viscosity, that is $F(a_t)$ is replaced by $\bar{F}(a_t)$ and we have
\begin{equation*}
{\mathrm{d}}_{\mathrm{t}}a_{t}=\bar{F}\left( a_{t}\right) ,
\end{equation*}%
where
\begin{equation*}
\bar{F}:=\left( 
\begin{array}{c}
v \\ 
h%
\end{array}%
\right) =\left( 
\begin{array}{c}
-u\cdot \nabla v-f\hat{z}\times u-\nabla p \\ 
-\nabla \cdot (hu)
\end{array}%
\right)
\end{equation*}%
and $\mathscr{G}:=L^{2}\left( \mathbb{T}^{2};\mathbb{R}^{2}\right)$,  $\mathscr{F}_{0}:=W^{1,2}\left( \mathbb{T}^{2};\mathbb{R}^{2}\right),\mathscr{F}_{1}:=W^{3,2}\left( \mathbb{T}^{2};\mathbb{R}^{2}\right)$ $\mathscr{D}:=W^{4,2}\left( \mathbb{T}^{2};\mathbb{R}^{2}\right)$ with similar inequalities as above.

\subsubsection{Inviscid incompressible models}
\noindent \textbf{Incompressible Euler equation in vorticity form} on
the 3D torus

\begin{equation}
d_{t}\omega +\mathcal{L}_{v}\omega =0  \label{eq Euler}
\end{equation}%
Here $\mathscr{G}:=L^{2}\left( \mathbb{T}^{3};\mathbb{R}%
^{3}\right) $,  $\mathscr{F}_{0}:=W^{\frac{3}{2}+\epsilon
,2}\left( \mathbb{T}^{3};\mathbb{R}^{3}\right) ,\mathscr{F}_{1}:=W^{3,2}\left( \mathbb{T}^{3};\mathbb{R}^{3}\right),$ $\mathscr{D}
:=W^{4,2}\left( \mathbb{T}^{3};\mathbb{R}^{3}\right).$ 
In this case, one can check that for $\omega \in \mathscr{D}:=W^{4,2}\left(\mathbb{T}^{3};\mathbb{R}^{3}\right) $ 
\begin{equation*}
\left\langle \omega ,A\left( \omega \right) \right\rangle _{W^{2,2}\left( \mathbb{T}^{3};\mathbb{R}^{3}\right) }\leq \alpha \left\vert
\left\vert \omega \right\vert \right\vert _{W^{\frac{3}{2}%
+\epsilon ,2}\left( \mathbb{T}^{3};\mathbb{R}^{3}\right) }\left\vert
\left\vert \omega \right\vert \right\vert _{W^{2,2}\left( \mathbb{T%
}^{3};\mathbb{R}^{3}\right) }^{2}
\end{equation*}%
Moreover%
\begin{equation*}
\left\vert \left\vert A\left( \omega \right) \right\vert \right\vert
_{L^{2}\left( \mathbb{T}^{3};\mathbb{R}^{3}\right) }\leq
c\left\vert \left\vert \omega \right\vert \right\vert _{W^{2,2}\left( \mathbb{T}^{3};\mathbb{R}^{3}\right) }^{\frac{3}{2}}
\end{equation*}%
which implies that $A$ is bounded on compact sets $K_{\epsilon }=\left\{
\Vert \omega \Vert _{W^{2,2}\left( \mathbb{T}^{3};\mathbb{R}%
^{3}\right) }\leq R\right\} \in L^{2}\left( \mathbb{T}^{3};\mathbb{%
R}^{3}\right) $. Also the interpolation property holds for $L^{2}\left( \mathbb{T}^{3};\mathbb{R}^{3}\right) ,$ $W^{\frac{3}{2}%
+\epsilon,2}\left( \mathbb{T}^{3};\mathbb{R}^{3}\right) $ and $W^{2,2}\left( \mathbb{T}^{3};\mathbb{R}^{3}\right) $. 

\vspace{5mm}
In general, global existence for the type of models presented in this section can be shown also in smoother spaces of the form $W^{k,2}(\mathbb{T}^2,\mathbb{R}^2)$ with $k\geq 2$ (see e.g. \cite{LangCrisan}, \cite{LangCrisanMemin}). This is useful for numerical purposes.

\subsection{Appendix}
\noindent\textbf{Aldous' criterion for tightness}

\noindent Let $(\mathbb{S}, \rho)$ be a separable and complete metric space. Let $\mathcal{D}(0, T ; \mathbb{S})$ denote the set of $\mathbb{S}$-valued càdlàg functions defined on $[0, T]$.
\begin{definition}
A sequence $\left\{X_n\right\}_{n=1}^{\infty}$ of  $\mathcal{D}(0, T ; \mathbb{S})$-valued random variables is said to satisfy the Aldous condition if and only if $\forall \epsilon>0, \forall \eta>0, \exists \delta>0$ such that for every sequence $\left\{\tau_n\right\}_{n=1}^{\infty}$ of $\left(\mathcal{F}_t\right)_{t \geq 0^{-}}$ - stopping times with $\tau_n \leq T$ we have
\begin{equation}\label{cond1}
\sup _{n \geq 1} \sup _{0 \leq t \leq \delta} \mathbb{P}\left\{\rho\left(X_n\left(\tau_n+t\right), X_n\left(\tau_n\right)\right) \geq \eta\right\} \leq \epsilon .
\end{equation}
\end{definition}
We can easily formulate a sufficient condition for \eqref{cond1} using Markov's inequality. Suppose that there exist constants $\alpha, \beta, C>0$ such that for every sequence $\left\{\tau_n\right\}_{n=1}^{\infty}$ of $\mathcal{F}_t$ - stopping times with $\tau_n \leq T$ we have
\begin{equation}\label{cond2}
\sup _{n \geq 1} \mathbb{E}\left(\rho\left(X_n\left(\tau_n+t\right), X_n\left(\tau_n\right)\right)^\alpha\right) \leq C t^\beta
\end{equation}
for every $t \geq 0$. Then $\left\{X_n\right\}_{n=1}^{\infty}$ satisfies the Aldous condition \eqref{cond1}. 

A sufficient condition for tightness of probability measures on $\mathcal{D}(0, T ; \mathbb{S})$: 

\begin{theorem}\label{aldouslemma}
Let $(\mathbb{S}, \rho)$ be a complete, separable metric space and let $\left\{X_n\right\}_{n=1}^{\infty}$ be a sequence of $\mathcal{D}(0, T ; \mathbb{S})$-valued random variables. If for every $t \in[0, T]$ the laws of $\left\{X_n(t)\right\}_{n=1}^{\infty}$ are tight on $\mathbb{S}$ and if condition \eqref{cond2} holds, then the laws of $\left\{X_n\right\}_{n=1}^{\infty}$ are tight on $\mathcal{D}(0, T ; \mathbb{S})$.
\end{theorem}

\begin{proposition}[\cite{RevuzYor}]\label{revuzyor}
If $Y$ is a continuous local martingale which vanishes at $0$, then
\begin{equation}
\mathbb{P}\left( \displaystyle\sup_{t}{Y}_t \geq x, \quad \langle Y\rangle_t \leq y \right) \leq \frac{1}{e^{-x^2/2y}}.
\end{equation}
\end{proposition}
%\OL{To check if Exercises 1.40, 1.41 page 134 are also helpful here}

\begin{proposition}[Lemma 2.5 in \cite{GawareckiMandrekar}]\label{gawareckimandrekar}
Let $K, H$ be separable Hilbert spaces, $Q$ a trace-class operator on $K$ and $\Phi \in L_2\left(K_Q, H\right)$. Then for arbitrary $\delta>0$ and $n>0$,
\begin{equation}
\mathbb{P}\left(\sup _{t \leq T}\left\|\int_0^t \Phi(s) d W_S\right\|_H>\delta\right) \leq \frac{n}{\delta^2}+\mathbb{P}\left(\int_0^T\|\Phi(s)\|_{\Lambda_2\left(K_Q, H\right)}^2 d s>n\right).
\end{equation}
\end{proposition}

\vspace{3mm}
\noindent\textbf{Funding}\\
\noindent Both authors have been supported by the European Research Council (ERC) under the European Union’s Horizon 2020 Research and Innovation Programme (ERC, Grant Agreement No 856408). 

\vspace{3mm}
\noindent\textbf{Data availability statement} \\
\noindent Data sharing not applicable to this article as no datasets were generated or analysed during the current study.

\vspace{3mm}
\noindent\textbf{Conflict of interest statement} \\
On behalf of both authors, the corresponding author states that there is no conflict of interest.

\end{document}